\documentclass[10pt,twoside]{siamart1116}

\usepackage{amsmath, fullpage,  url, float, subcaption, rotating, fp,  tikz, enumerate,mathrsfs}
\usetikzlibrary{graphs,graphs.standard}\usepackage{dsfont} 
\usepackage{tkz-berge,tkz-graph}
\usepackage{multicol}
\usepackage[mathlines]{lineno}

\usepackage[english]{babel}
\usepackage{graphicx,epstopdf,epsfig}
\usepackage{amsfonts,epsfig,fancyhdr,graphics, hyperref,amsmath,amssymb}
\usepackage{bbm}

\renewcommand{\theequation}{\thesection.\arabic{equation}}
\renewtheorem{theorem}{Theorem}[section]

\newtheorem{thm}{Theorem}[section]
\newtheorem{cor}[thm]{Corollary}
\newtheorem{lem}[thm]{Lemma}
\newtheorem{prop}[thm]{Proposition}
\newtheorem{conj}[thm]{Conjecture}
\newtheorem{obs}[thm]{Observation}

\newtheorem{defn}[thm]{Definition}

\numberwithin{figure}{section}  
\numberwithin{table}{section}

\definecolor{red}{rgb}{1,0,0}
\def\red{\color{red!70}}

\newcommand{\R}{\mathbb{R}}

\newcommand{\spec}{\operatorname{spec}}

\newcommand{\trace}{\operatorname{trace}}
\newcommand{\sgn}{\operatorname{sgn}}

\newcommand{\bit}{\begin{itemize}}
\newcommand{\eit}{\end{itemize}}
\newcommand{\ben}{\begin{enumerate}}
\newcommand{\een}{\end{enumerate}}
\newcommand{\beq}{\begin{equation}}
\newcommand{\eeq}{\end{equation}}
\newcommand{\bea}{\begin{eqnarray*}} 
\newcommand{\eea}{\end{eqnarray*}}
\newcommand{\bpf}{\begin{proof}}
\newcommand{\epf}{\end{proof}\ms}
\newcommand{\bmt}{\begin{bmatrix}}
\newcommand{\emt}{\end{bmatrix}}
\newcommand{\ms}{\medskip}

\newcommand{\CP}{\mathcal{CP}}
\newcommand{\Cy}{\mathcal{C}}

\newcommand{\D}{\mathcal D}
\newcommand{\cL}{\mathcal L}
\newcommand{\DL}{\D^{L}}
\newcommand{\DQ}{\D^{Q}}
\newcommand{\NDL}{\D^{\cL}}
\newcommand{\tr}{\operatorname{t}}
\newcommand{\tm}{t_{\text{min}}}
\newcommand{\tM}{t_{\text{max}}}
\newcommand{\diag}{\operatorname{diag}}

\title{The normalized distance Laplacian}
\author{Carolyn Reinhart\thanks{Department of Mathematics, Iowa State University, Ames, IA 50011, USA (reinh196@iastate.edu) }}
\begin{document}
\maketitle
\begin{abstract}
The distance matrix $\mathcal{D}(G)$ of a graph $G$ is the matrix containing the pairwise distances between vertices. The transmission of a vertex $v_i$ in $G$ is the sum of the distances from $v_i$ to all other vertices and $T(G)$ is the diagonal matrix of transmissions of the vertices of the graph. The normalized distance Laplacian, $\mathcal{D}^{\mathcal {L}}(G)=I-T(G)^{-1/2}\mathcal{D}(G) T(G)^{-1/2}$, is introduced. This is analogous to the normalized Laplacian matrix, $\mathcal{L}(G)=I-D(G)^{-1/2}A(G)D(G)^{-1/2}$, where $D(G)$ is the diagonal matrix of degrees of the vertices of the graph and $A(G)$ is the adjacency matrix. Bounds on the spectral radius of $\mathcal{D}^{\mathcal {L}}$ and connections with the normalized Laplacian matrix are presented. Twin vertices are used to determine eigenvalues of the normalized distance Laplacian. The distance generalized characteristic polynomial is defined and its properties established. Finally, $\mathcal{D}^{\mathcal {L}}$-cospectrality and lack thereof is determined for all graphs on 10 and fewer vertices, providing evidence that the normalized distance Laplacian has fewer cospectral pairs than other matrices.
\end{abstract}

\section{Introduction}

Spectral graph theory is the study of matrices defined in terms of a graph, specifically relating the eigenvalues of the matrix to properties of the graph. Many such matrices are studied and they can often be used in applications. The normalized Laplacian, a matrix popularized by Fan Chung in her book Spectral Graph Theory, has applications in random walks \cite{Chu97}. The distance matrix was defined by Graham and Pollak in \cite{GP71} in order to study the problem of loop switching in routing telephone calls through a network. In this paper, we introduce the normalized distance Laplacian, which is defined analogously to the normalized Laplacian but incorporates distances between each pair of vertices in the graph.

A {\em weighted graph} is a graph with vertices $V$ and a weight function $w$ that assigns a positive real number to each edge in the graph. The {\em degree} of a vertex in a weighted graph is the sum of the weights of the edges incident to it. Any unweighted graph $G$ may be seen as a weighted graph with edge weights equal to 1.

The {\em adjacency matrix} of a weighted graph $G$ is the real symmetric matrix defined by $(A(G))_{ij}= w(v_i,v_j)$.
The eigenvalues of $A(G)$ are called the {\em adjacency eigenvalues} and are denoted $\lambda_1\geq \dots\geq \lambda_n$ (note that while adjacency eigenvalues are ordered largest to smallest, Laplacian eigenvalues will be ordered smallest to largest).
The {\em degree matrix} is the diagonal matrix $D(G)=\diag(\deg(v_1),\dots,\deg(v_n))$.
The {\em combinatorial Laplacian matrix} of a weighted graph $G$, denoted $L(G)$, has entries \[(L(G))_{ij}= \begin{cases}
 -w(v_i,v_j) &  i\not=j \\
 \deg(v_i) &  i=j
 \end{cases}\]
and it is easy to observe that $L(G)=D(G)-A(G)$. The eigenvalues of $L(G)$ are called the {\em combinatorial Laplacian eigenvalues} and are denoted $\phi_1\leq \dots\leq \phi_n$. Since $L(G)$ is a positive semi-definite matrix, $\rho_{L(G)}=\phi_n$. The matrix $Q(G)=D(G)+A(G)$ is called the {\em signless Laplacian}. The eigenvalues of $Q(G)$ are called the {\em signless Laplacian eigenvalues} and are denoted $q_1\leq \dots\leq q_n$.
The {\em normalized Laplacian matrix} of a weighted graph $G$ without isolated vertices, denoted $\cL(G)$, has entries \[(\cL(G))_{ij}= \begin{cases}
 -\frac{w(v_i,v_j)}{\sqrt{\deg(v_i)\deg(v_j)}} &  i\not=j \\
 1 &  i=j\\
 \end{cases}.\]
 Observe that $\cL(G)=D(G)^{-1/2}L(G)D(G)^{-1/2}=I-D(G)^{-1/2}A(G)D(G)^{-1/2}$. The eigenvalues of $\cL(G)$ are called the {\em normalized Laplacian eigenvalues} and are denoted $\mu_1\leq \dots\leq \mu_n$. Since $\cL(G)$ is a positive semi-definite matrix, $\rho_{\cL(G)}=\mu_n$.
 
 The four matrices are also denoted just $A,L,Q,$ and $\cL$ when the intended graph is clear. Note that while all these matrices are defined for graphs in general, in this paper we consider them for connected graphs only, unless otherwise stated.


The {\em distance matrix}, denoted $\D(G)$, has entries $(\D(G))_{ij}=d(v_i,v_j)$ where $d(v_i,v_j)$ is the {\em distance} (number of edges in a shortest path) between $v_i$ and $v_j$. Much work has been done to study the spectra of distance matrices; for a survey see $\cite{AH14}$. Requiring that every graph $G$ be connected ensures that $d(v_i,v_j)$ is finite for every pair of vertices $v_i,v_j\in V(G)$. The eigenvalues of $\D(G)$ are called {\em distance eigenvalues} and are denoted $\partial_1\geq \dots\geq \partial_n$ (note that while distance eigenvalues are ordered largest to smallest, distance Laplacian eigenvalues will be ordered smallest to largest).
In a graph $G$, the {\em transmission} of a vertex $v\in V(G)$, denoted $\tr_G(v)$ or $\tr(v)$ when the intended graph is clear, is defined as $\tr_G(v)=\sum_{u_i\in V(G)} d(v,u_i)$. A graph is {\em $k$-transmission regular} if $\tr(v)=k$ for all $v\in V$. The {\em transmission matrix} is the diagonal matrix $T(G)= \diag(\tr(v_1),\dots,\tr(v_n))$.

In \cite{AH13}, Aouchiche and Hansen defined the distance Laplacian and the signless distance Laplacian. The {\em distance Laplacian matrix}, denoted $\DL(G)$, has entries \[(\DL(G))_{ij}= \begin{cases}
 -d(v_i,v_j) &  i \neq j \\
 \tr(v_i) &  i=j
 \end{cases}\]
and $\DL(G)=T(G)-\D(G)$. The eigenvalues of $\DL(G)$ are called {\em distance Laplacian eigenvalues} and are denoted $\partial^{L}_1\leq \dots\leq \partial^{L}_n$. Since $\DL(G)$ is a positive semi-definite matrix, $\rho_{\DL(G)}=\partial^{L}_n$. The matrix $\DQ(G)=T(G)+\D(G)$ is called the {\em signless distance Laplacian}. The eigenvalues of $\DQ(G)$ are called the {\em signless distance Laplacian eigenvalues} and are denoted $\partial^{Q}_1\leq \dots\leq \partial^{Q}_n$. The matrices are also denoted as just $\D,\DL,$ and $\DQ$ when the intended graph is clear.

In Section \ref{sec:NormDistLap}, we define the normalized distance Laplacian and show that its spectral radius is strictly less than 2, in contrast with the normalized Laplacian whose spectral radius is equal to 2 when the graph is bipartite. We also find bounds on the normalized distance Laplacian eigenvalues and provide data that leads to conjectures about the graphs achieving the maximum and minimum spectral radius. Methods using twin vertices to determine eigenvalues for the normalized distance Laplacian are described in Section \ref{sec:EigenDet} and applied to determine the spectrum of several families of graphs. In Section \ref{sec:GenDistChar}, we define the generalized distance generalized characteristic polynomial. We show that if the polynomial is equal for two non-isomorphic graphs, they have the same $\D$, $\DL$, $\DQ$, and $\NDL$ spectra and the same multiset of transmissions, extending concepts from the generalized characteristic polynomial.

Two non-isomorphic graphs $G$ and $H$ are {\em $M$-cospectral} if $\spec(M(G))=\spec(M(H))$; if $G$ and $H$ are $M$-cospectral we call them {\em M-cospectral mates} (or just cospectral mates if the choice of $M$ is clear). A graph parameter is said to be {\em preserved by $M$-cospectrality} if two graphs that are $M$-cospectral must share the same value for that parameter (can be numeric or true/false). Cospectral graphs and the preservation of parameters has been studied for many different matrices. Godsil and McKay were the first to produce an adjacency cospectrality construction \cite{GM82} but many other papers study cospectrality of the normalized Laplacian (see, for example, 
\cite{But15},\cite{BG11},\cite{BH16},\cite{O13}). Several of these papers also discuss preservation by $M$-cospectrality; for a table summarizing preservation by $A,L,Q$ and $\cL$ cospectrality of some well known graph parameters, see \cite{BC}. 

Cospectrality of $\D,\DL,$ and $\DQ$ was studied by Aouchiche and Hansen in \cite{AH18} and cospectral constructions have been found for the distance matrix in \cite{H17}. Cospectral constructions for the distance Laplacian matrix were exhibited in \cite{BDHLRSY19} and several graph parameters were shown to not be preserved by $\DL$-cospectrality. In Section \ref{sec:Cospec}, we find all cospectral graphs on 10 or fewer vertices for the normalized distance Laplacian and show how some of the graph pairs could be constructed using $\DL$-cospectrality constructions. We also use examples of graphs on 9 and 10 vertices to show several parameters are not preserved by normalized distance Laplacian cospectrality and provide evidence that cospectral mates are rare for this matrix. Since graphs with different spectra cannot possibly be isomorphic, the spectrum of graphs with respect to matrices can be thought of as a tool to differentiate between graphs. Because of this, the rarity of cospectrality is beneficial. 

Throughout the paper, we use the following standard definitions and notation. A {\em graph} $G$ is a pair $G=(V,E)$, where $V=\{v_1,\dots,v_n\}$ is the set of {\em vertices} and $E$ is the set of edges. An {\em edge} is a two element subset of vertices $\{v_i,v_j\}$, also denoted as just $v_iv_j$. We use $n=|V|$ to denote the {\em order} of $G$ and assume all graphs $G$ are connected and simple (i.e. no loops or multiedges). Two vertices $v_i$ and $v_j$ are {\em neighbors} if $v_iv_j\in E(G)$ and the {\em neighborhood} $N(v)$ of a vertex $v$ is the set of its neighbors. The {\em degree} of a vertex $v$ is $\deg(v)=|N(v)|$. A graph is {\em $k$-regular} if $\deg(v_i)=k$ for all $1\leq i\leq n$. Let $\spec(M)$ denote the {\em spectrum} of a matrix $M$ and let $p_M(x)$ denote the {\em characteristic polynomial} of matrix $M$. The {\em spectral radius} of a matrix $M$ with eigenvalues $\nu_1\leq\dots\leq\nu_n$ is $\rho_{M}=\max_{1\leq i\leq n} |\nu_i|$. An $n\times n$ real symmetric matrix $M$ is {\em positive semi-definite} if $x^TMx\geq 0$ for all $x\in\R^n$. Equivalently, a real symmetric matrix $M$ is positive semi-definite if and only if all its eigenvalues are non-negative. If all eigenvalues are non-negative, observe $\rho_M=\nu_n$. Note all matrices we will consider are real and symmetric.

\section{The normalized distance Laplacian}\label{sec:NormDistLap}
As with the combinatorial Laplacian matrix, it is natural to define a normalized version of the distance Laplacian matrix. In this section, we introduce this new matrix and derive many proprieties of its eigenvalues.
\begin{defn}
The {\em normalized distance Laplacian matrix}, denoted $\NDL(G)$, or just $\NDL$, is the matrix with entries \[(\NDL(G))_{ij}= \begin{cases}
 \frac{-d(v_i,v_j)}{\sqrt{\tr(v_i)\tr(v_j)}} &  i \neq j \\
 1 &  i=j
 \end{cases}.\]
\end{defn}
Observe that $\NDL(G)=T(G)^{-1/2}\DL(G)T(G)^{-1/2}=I-T(G)^{-1/2}\D(G)T(G)^{-1/2}$. We call the eigenvalues of $\NDL(G)$ the {\em normalized distance Laplacian eigenvalues} and denote them $\partial^{\cL}_1\leq \dots\leq \partial^{\cL}_n$.

It is easy to draw parallels between the properties of $A, L, Q,$ and $\cL$ and the properties of $\D$, $\DL, \DQ,$ and $\NDL$. In the remainder of this section, we present results that are known to hold for the adjacency matrix and its Laplacians, followed by their generalizations to the distance matrices.

Both the normalized Laplacian and the normalized distance Laplacian include square roots (unless the graph is regular or transmission regular, respectively). This can make computation of eigenvalues difficult with these matrices. Because of this, we can turn to similar matrices that make computation slightly easier. In \cite{Chu97}, Chung introduces the matrix $D^{-1} L$, which one can easily see is similar to $\cL$ by the similarity matrix $D^{-1/2}$. The eigenvectors $\mathbf{v_i}$ of $D^{-1} L(G)$ are called the {\em harmonic eigenvectors of $\cL(G)$} and $\mathbf{v_i}=D^{-1/2}\mathbf{u_i}$ where $\mathbf{u_i}$ is an eigenvector of $\cL$. We now show an analogous similar matrix for $\NDL$.

\begin{prop}
For all eigenvalues $\partial_i^{\cL}$ of $\NDL$ and associated eigenvectors $\mathbf{x_i}$, $\partial_i^{\cL}$ is also an eigenvalue of $T^{-1}\DL$ with associated eigenvector $\mathbf{y_i}=T^{-1/2}\mathbf{x_i}$.
\end{prop}
\bpf
\begin{eqnarray*}
T^{-1}\DL \mathbf{y}_i&=&T^{-1}\DL T^{-1/2}\mathbf{x_i}\\
&=&T^{-1/2}\NDL\mathbf{x_i}\\
&=&T^{-1/2}\partial_i^{\cL}\mathbf{x_i}\\
&=&\partial_i^{\cL}\mathbf{y_i}
\end{eqnarray*}
So $\partial_i^{\cL}$ is an eigenvalue of $T^{-1}\DL$ with associated eigenvector $\mathbf{y_i}$, as desired.
\epf

Call the eigenvectors $\mathbf{y_i}$ of $T^{-1} \DL(G)$ the {\em harmonic eigenvectors of $\NDL(G)$}.

The following relationship between the eigenvalues of $A(G)$ and $\cL(G)$ can be observed using Sylvester's law of inertia. 
 
 \begin{prop}{\em \cite[p. 14]{But08}}\label{mult}
 The multiplicity of $0$ as an eigenvalue of $A(G)$ is the multiplicity of $1$ as an eigenvalue of $\cL(G)$, the number of negative eigenvalues for $A(G)$ is the number of eigenvalues greater than $1$ for $\cL(G)$, and the number of positive eigenvalues for $A(G)$ is the number of eigenvalues less than $1$ for $\cL(G)$.
 \end{prop}

The analogous result for $\D$ and $\NDL$ can be shown using the proof technique suggested by Butler in \cite{But08}. Two matrices $A$ and $B$ are {\em congruent} if there exists an invertible matrix $P$ such that $P^TAP=B$. Sylvester's law of inertia states that any two real symmetric matrices that are congruent have the same number of positive, negative, and zero eigenvalues.
\begin{prop}
The multiplicity of $0$ as an eigenvalue of $\D(G)$ is the multiplicity of $1$ as an eigenvalue of $\NDL(G)$, the number of negative eigenvalues for $\D(G)$ is the number of eigenvalues greater than $1$ for $\NDL(G)$, the number of positive eigenvalues for $\D(G)$ is the number of eigenvalues greater than $1$ for $\NDL(G)$.
\end{prop}
\bpf
Since $(T(G)^{-1/2})^T=T(G)^{-1/2}$, $\D(G)$ is congruent to $T(G)^{-1/2}\D(G)T(G)^{-1/2}$, and therefore they have the same number of positive, negative, and zero eigenvalues. It is easy to see $0$ is an eigenvalue of $T(G)^{-1/2}\D(G)T(G)^{-1/2}$ if and only if $1$ is an eigenvalue of $\NDL(G)$. If $\nu< 0$ is an eigenvalue of $T(G)^{-1/2}\D(G)T(G)^{-1/2}$, then $1-\nu> 1$ is an eigenvalue of $\NDL(G)$. Similarly, if $\nu> 0$ is an eigenvalue of $T(G)^{-1/2}\D(G)T(G)^{-1/2}$, then $1-\nu< 1$ is an eigenvalue of $\NDL(G)$.
\epf

In special cases, we may deduce an exact relationship between the eigenvalues of various matrices. The following fact is easy to observe and well-known in the literature.

\begin{obs}\label{Obs:Reg} For a $r$-regular weighted graph, $D(G)=rI$ so for every adjacency eigenvalue $\lambda_i$, $\phi_i=r-\lambda_i$, $q_i=r+\lambda_i$, and $\mu_i=1-\frac{1}{r}\lambda_{i}$. Similarly, for a $k$-transmission regular graph, $T(G)=kI$ so for every distance eigenvalue $\partial_i$, $\partial^L_i=k-\partial_i$ and $\partial^{Q}_i=k+\partial_i$.
\end{obs}

For $k$-transmission regular graphs, the relationships between the eigenvalues of $\NDL$ and $\D,\DL,\DQ$ are also easily observed.

\begin{obs}\label{Obs:TransReg}
For a $k$-transmission regular graph $G$, the normalized distance Laplacian eigenvalues are $\partial^{\cL}_{i}=\frac{1}{k}\partial^{L}_{i}=1-\frac{1}{k}\partial_{i}=2-\frac{1}{k}\partial^{Q}_{i}$.
\end{obs}

This observation can be applied to compute the $\NDL$-spectrum for some transmission regular graph families. The spectrum of $\DL(K_n)$ is $\{0,n^{(n-1)}\}$ \cite{AH13} and the complete graph is $n-1$-transmission regular, so it is easy to observe $\spec(\NDL(K_n))=\{0,\frac{n}{n-1}^{(n-1)}\}$.

In \cite{AH13}, the distance Laplacian eigenvalues are given for a cycle. For even length cycles where $n=2p$, \[\spec(\DL(C_n))=\left\{0, \left(\frac{n^2}{4}\right)^{(p-1)},\frac{n^2}{4}+\csc^2\left(\frac{\pi (2j-1)}{n}\right)\right\}\, \text{for} \, j=1,\dots,p\]
and for odd length cycles where $n=2p+1$, \[\spec(\DL(C_n))=\left\{0,\frac{n^2-1}{4}+\frac{1}{4}\sec^2\left(\frac{\pi j}{n}\right), \frac{n^2-1}{4}-\frac{1}{4}\sec^2\left(\frac{\pi (2j-1)}{2n}\right)\right\} \, \text{for} \, j=1,\dots,p.\]

The cycle is a transmission regular graph with transmission $\frac{n^2}{4}$ when $n$ is even and transmission $\frac{n^2-1}{4}$ when $n$ is odd. So we can apply Observation \ref{Obs:TransReg} to these known spectra to obtain the eigenvalues of $\NDL(C_n)$.
\begin{prop}
Let $C_n$ be the cycle on $n$ vertices. Then if $n=2p$ is even, \[\spec(\NDL(C_n))=\left\{0, 1^{(p-1)},1+\frac{4}{n^2}\csc^2\left(\frac{\pi (2j-1)}{n}\right)\right\}\, \text{for} \, j=1,\dots,p\]
and if $n=2p+1$ is odd, \[\spec(\NDL(C_n))=\left\{0,1+\frac{1}{n^2-1}\sec^2\left(\frac{\pi j}{n}\right), 1-\frac{1}{n^2-1}\sec^2\left(\frac{\pi (2j-1)}{2n}\right)\right\} \, \text{for} \, j=1,\dots,p.\]

\end{prop}

In her book that describes the normalized Laplacian matrix \cite{Chu97}, Chung finds many bounds on the eigenvalues of $\cL$. We now show similar results hold for the eigenvalues of the normalized distance Laplacian. The first result provides a range in which all eigenvalues of $\cL$ lie and notes that both bounds are achieved. The result appears with proof in \cite{Chu97} and is stated without proof for weighted graphs in \cite{BC}; one can verify the proof from \cite{Chu97} remains valid for weighted graphs.

 \begin{thm}{\em \cite{Chu97}}\label{ChuBnds}
For all weighted connected graphs $G$, \[0=\mu_1< \mu_2\leq \dots\leq \mu_n\leq 2,\]
with $\mu_n=2$ if and only if $G$ is non-trivial and bipartite.
 \end{thm}

This result generalizes with one notable difference: The normalized distance Laplacian never achieves 2 as an eigenvalue for $n\geq 3$. Observe the normalized distance Laplacian of a graph $G$ is the normalized Laplacian of the weighted complete graph $W(G)$ with edges weights $w_{W(G)}(v_i,v_j)=d_G(v_i,v_j)$. Then the degree of a vertex $v_i$ in $W(G)$ is the transmission of $v_i$ in $G$. The next result is an application of Theorem \ref{ChuBnds} to $W(G)$ along with the observation that complete graphs on $n\geq 3$ vertices are not bipartite.

\begin{cor}
For all graphs $G$, \[0=\partial^{\cL}_1< \partial^{\cL}_2\leq \dots\leq \partial^{\cL}_n=\rho_{\NDL}\leq 2,\]
and for $n\geq 3$, $\partial^{\cL}_n< 2$.
\end{cor}

Since this bound is not tight, a natural next question is: Which graphs have the largest spectral radius? Using a {\em Sage} search \cite{Sage1}, the maximum spectral radius of any graph on a given numbers of vertices was determined for $n\leq 10$. The results are listed in Table \ref{tab:MaxRho} below. Define $KPK_{n_1,n_2,n_3}$ for $n_1,n_3\geq 1$, $n_2\geq 2$ to be the graph formed by taking the vertex sum of a vertex in $K_{n_1}$ with one end of the path $P_{n_2}$ and the vertex sum of a vertex in $K_{n_3}$ with the other end of $P_{n_2}$. Note the number of vertices is $n=n_1+n_2+n_3-2$ and $KPK_{1,n,1}=KPK_{2,n-1,1}=KPK_{2,n-2,2}=P_{n}$.

\begin{figure}[h!]
\begin{center}
\begin{tikzpicture}[scale=0.7]
\definecolor{cv0}{rgb}{0.0,0.0,0.0}
\definecolor{cfv0}{rgb}{1.0,1.0,1.0}
\definecolor{clv0}{rgb}{0.0,0.0,0.0}
\definecolor{cv1}{rgb}{0.0,0.0,0.0}
\definecolor{cfv1}{rgb}{1.0,1.0,1.0}
\definecolor{clv1}{rgb}{0.0,0.0,0.0}
\definecolor{cv2}{rgb}{0.0,0.0,0.0}
\definecolor{cfv2}{rgb}{1.0,1.0,1.0}
\definecolor{clv2}{rgb}{0.0,0.0,0.0}
\definecolor{cv3}{rgb}{0.0,0.0,0.0}
\definecolor{cfv3}{rgb}{1.0,1.0,1.0}
\definecolor{clv3}{rgb}{0.0,0.0,0.0}
\definecolor{cv4}{rgb}{0.0,0.0,0.0}
\definecolor{cfv4}{rgb}{1.0,1.0,1.0}
\definecolor{clv4}{rgb}{0.0,0.0,0.0}
\definecolor{cv5}{rgb}{0.0,0.0,0.0}
\definecolor{cfv5}{rgb}{1.0,1.0,1.0}
\definecolor{clv5}{rgb}{0.0,0.0,0.0}
\definecolor{cv6}{rgb}{0.0,0.0,0.0}
\definecolor{cfv6}{rgb}{1.0,1.0,1.0}
\definecolor{clv6}{rgb}{0.0,0.0,0.0}
\definecolor{cv7}{rgb}{0.0,0.0,0.0}
\definecolor{cfv7}{rgb}{1.0,1.0,1.0}
\definecolor{clv7}{rgb}{0.0,0.0,0.0}
\definecolor{cv8}{rgb}{0.0,0.0,0.0}
\definecolor{cfv8}{rgb}{1.0,1.0,1.0}
\definecolor{clv8}{rgb}{0.0,0.0,0.0}
\definecolor{cv0v3}{rgb}{0.0,0.0,0.0}
\definecolor{cv0v6}{rgb}{0.0,0.0,0.0}
\definecolor{cv0v8}{rgb}{0.0,0.0,0.0}
\definecolor{cv1v4}{rgb}{0.0,0.0,0.0}
\definecolor{cv1v7}{rgb}{0.0,0.0,0.0}
\definecolor{cv2v5}{rgb}{0.0,0.0,0.0}
\definecolor{cv2v7}{rgb}{0.0,0.0,0.0}
\definecolor{cv3v6}{rgb}{0.0,0.0,0.0}
\definecolor{cv3v8}{rgb}{0.0,0.0,0.0}
\definecolor{cv4v7}{rgb}{0.0,0.0,0.0}
\definecolor{cv5v8}{rgb}{0.0,0.0,0.0}
\definecolor{cv6v8}{rgb}{0.0,0.0,0.0}
\Vertex[style={minimum
size=1.0cm,draw=cv0,fill=cfv0,text=clv0,shape=circle},LabelOut=false,L=\hbox{$1$},x=0.0cm,y=0cm]{v0}
\Vertex[style={minimum
size=1.0cm,draw=cv1,fill=cfv1,text=clv1,shape=circle},LabelOut=false,L=\hbox{$9$},x=9cm,y=1.0cm]{v1}
\Vertex[style={minimum
size=1.0cm,draw=cv2,fill=cfv2,text=clv2,shape=circle},LabelOut=false,L=\hbox{$6$},x=6cm,y=0cm]{v2}
\Vertex[style={minimum
size=1.0cm,draw=cv3,fill=cfv3,text=clv3,shape=circle},LabelOut=false,L=\hbox{$2$},x=1cm,y=-1cm]{v3}
\Vertex[style={minimum
size=1.0cm,draw=cv4,fill=cfv4,text=clv4,shape=circle},LabelOut=false,L=\hbox{$8$},x=9cm,y=-1cm]{v4}
\Vertex[style={minimum
size=1.0cm,draw=cv5,fill=cfv5,text=clv5,shape=circle},LabelOut=false,L=\hbox{$5$},x=4cm,y=0cm]{v5}
\Vertex[style={minimum
size=1.0cm,draw=cv6,fill=cfv6,text=clv6,shape=circle},LabelOut=false,L=\hbox{$3$},x=1cm,y=1cm]{v6}
\Vertex[style={minimum
size=1.0cm,draw=cv7,fill=cfv7,text=clv7,shape=circle},LabelOut=false,L=\hbox{$7$},x=8cm,y=0cm]{v7}
\Vertex[style={minimum
size=1.0cm,draw=cv8,fill=cfv8,text=clv8,shape=circle},LabelOut=false,L=\hbox{$4$},x=2cm,y=0cm]{v8}
\Edge[lw=0.1cm,style={color=cv0v3,},](v0)(v3)
\Edge[lw=0.1cm,style={color=cv0v6,},](v0)(v6)
\Edge[lw=0.1cm,style={color=cv0v8,},](v0)(v8)
\Edge[lw=0.1cm,style={color=cv1v4,},](v1)(v4)
\Edge[lw=0.1cm,style={color=cv1v7,},](v1)(v7)
\Edge[lw=0.1cm,style={color=cv2v5,},](v2)(v5)
\Edge[lw=0.1cm,style={color=cv2v7,},](v2)(v7)
\Edge[lw=0.1cm,style={color=cv3v6,},](v3)(v6)
\Edge[lw=0.1cm,style={color=cv3v8,},](v3)(v8)
\Edge[lw=0.1cm,style={color=cv4v7,},](v4)(v7)
\Edge[lw=0.1cm,style={color=cv5v8,},](v5)(v8)
\Edge[lw=0.1cm,style={color=cv6v8,},](v6)(v8)
\end{tikzpicture}
\end{center}
\caption{$KPK_{4,4,3}$}\label{Fig:KPK}
\end{figure}
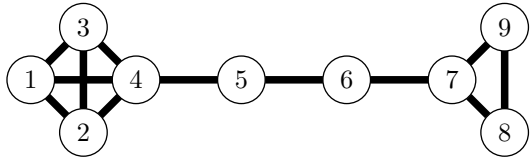

\begin{table}[h!]
\begin{center}
\begin{tabular}{|c|c|c|}
    n & $\rho_{\NDL}$ & Graph  \\
    \hline
    2 & 2 & $KPK_{1,2,1}$ \\
    3 & 1.666 & $KPK_{2,2,1}$\\
    4 & 1.614 & $KPK_{2,2,2}$\\
    5 & 1.589 & $KPK_{2,3,2}$\\
    6 & 1.578 & $KPK_{3,3,2}$\\
    7 & 1.586 & $KPK_{3,3,3}$\\
    8 & 1.590 & $KPK_{3,4,3}$\\
    9 & 1.594 & $KPK_{4,4,3}$\\
    10 & 1.603 & $KPK_{4,4,4}$\\
\end{tabular}
\caption{The maximum $\rho_{\NDL}$ and graph that achieves it for all graphs on 10 or fewer vertices}
\label{tab:MaxRho}
\end{center}
\end{table}

It is natural to conjecture a pattern from the graphs given in Table \ref{tab:MaxRho}. For example, for $n=3\ell$, one might conjecture that the graph achieving the maximum value of $\rho_{\NDL}$ is $KPK_{\ell+1,\ell+1,\ell}$. However, as $n$ grows larger, this pattern does not hold. For example, when $n=15$, $\rho_{\NDL}(KPK_{6,5,6})>\rho_{\NDL}(KPK_{6,6,5})$. In Table \ref{tab:LargeRho} we provide evidence that $\rho_{\NDL}(KPK_{n_1,n_2,n_3})$ tends towards 2 as $n$ becomes large for some $n_1,n_2,n_3$. Note that these graphs were the graphs with largest $\rho_{\NDL}$ found by checking several graphs in the $KPK_{n_1,n_2,n_3}$ family on {\em Sage}, and are not guaranteed to have the largest $\rho_{\NDL}$ of all graphs on $n$ vertices or even within the family $KPK_{n_1,n_2,n_3}$. This data leads to the next conjecture.

\begin{table}[h!]
\begin{center}
\begin{tabular}{|c|c|c|}
    n & $\rho_{\NDL}$ & Graph  \\
    \hline
    15 & 1.634 & $KPK_{6,5,6}$\\
    20 & 1.661 & $KPK_{8,6,8}$\\
    25 & 1.682 & $KPK_{10,7,10}$\\
    50& 1.748 & $KPK_{21,10,21}$\\
    100& 1.808 & $KPK_{43,16,43}$\\
    200 & 1.857 & $KPK_{90,22,90}$\\
    400 & 1.895 & $KPK_{184,34,184}$\\
    600 & 1.913 & $KPK_{280,42,280}$\\
    800 & 1.924 & $KPK_{377,48,377}$\\
\end{tabular}
\caption{Evidence that maximum $\rho_{\NDL}$ approaches 2 as $n$ becomes large, data from {\em Sage} \cite{Sage1} }
\label{tab:LargeRho}
\end{center}
\end{table}

\begin{conj}
The maximum $\NDL$ spectral radius achieved by a graph on $n$ vertices tends to $2$ as $n\to\infty$ and is achieved by $KPK_{n_1,n_2,n_3}$ for some $n_1+n_2+n_3=n+2$.
\end{conj}

This family also shows that while $\rho_{\DL}$ is subgraph monotonically increasing (see \cite[Theorem 3.5]{AH13}), $\rho_{\NDL}$ is not. Specifically, we can see that $P_n$ is a subgraph of $KPK_{n_1,n_2,n_3}$ for all $n_1+n_2+n_3=n+2$. However, it has been verified using {\em Sage} \cite{Sage1} for $n\leq 20$ that $\rho_{\NDL}(P_n)< \rho_{\NDL}(KPK_{n_1,n_2,n_3})$ for some $n_1+n_2+n_3=n+2$.

The following result provides a bound for the smallest non-zero eigenvalue and the largest eigenvalue with respect to $\cL$.

\begin{thm}\textnormal{\cite[Lemma 1.7(ii)]{Chu97}}\label{ChuTraceBnds}
 For all graphs $G$ on $n\geq 2$ vertices,
 \[\mu_2\leq \frac{n}{n-1} \] with equality holding if and only if $G$ is $K_n$. Also, \[\rho=\mu_n\geq \frac{n}{n-1} .\]
 \end{thm}

The proof of this result can be used to prove nearly the same result for $\NDL$. Note the proof of equality of the first inequality if and only if $G$ is $K_n$ could not be generalized, since the proof relies on $\cL$ having 0 entries corresponding to non-adjacencies.

\begin{thm}\label{minSpecRadius}
 For a graph $G$ on $n\geq 2$ vertices, \[\partial^{\cL}_2\leq \frac{n}{n-1} \, \text{ and }\, \rho_{\NDL}=\partial^{\cL}_n\geq \frac{n}{n-1}.\] 
 \end{thm}
\bpf
Observe $\sum\limits_{i=1}^n \partial^{\cL}_i=\sum\limits_{i=2}^n \partial^{\cL}_i=\trace(\NDL)=n$. 
Then since $\partial^{\cL}_2$ is the smallest non-zero eigenvalue, $\partial^{\cL}_2(n-1)\leq \sum\limits_{i=2}^n \partial^{\cL}_i= n$ so $\partial^{\cL}_2\leq \frac{n}{n-1}$. 
Similarly, since $\partial^{\cL}_n$ is the largest eigenvalue $\partial^{\cL}_n(n-1)\geq \sum\limits_{i=2}^n \partial^{\cL}_i= n$ so $\partial^{\cL}_n\geq \frac{n}{n-1}$.
\epf

We can see that Theorem \ref{minSpecRadius} provides a lower bound on the spectral radius of $\NDL$. As previously computed, this is the spectral radius of the complete graph, so this minimum is achieved by $K_n$. In fact, we can prove the following stronger statement.

\begin{thm}\label{minSpec}
If any graph $G$ has $\NDL$ spectral radius $\frac{n}{n-1}$, $\spec_{\NDL}(G)=\{0,\frac{n}{n-1}^{(n-1)}\}$.
\end{thm}

\bpf
For a graph $G$, let $\rho_{\NDL}=\partial^{\cL}_n= \frac{n}{n-1}$. Recall $\partial^{\cL}_1=0$ for all graphs and obviously $\partial^{\cL}_2\leq \dots\leq \partial^{\cL}_{n-1}\leq \frac{n}{n-1}$ by definition. As in the proof of the previous theorem, we have $\sum\limits_{i=2}^n \partial^{\cL}_i=\trace(\NDL)=n$ so $\sum\limits_{i=2}^{n-1} \partial^{\cL}_i=n-\frac{n}{n-1}=\frac{n(n-2)}{n-1}$. If $\partial^{\cL}_i<\frac{n}{n-1}$ for some $2\leq i\leq n-1$, $\sum\limits_{i=2}^{n-1} \partial^{\cL}_i< \frac{n(n-2)}{n-1}$. So $\partial^{\cL}_i=\frac{n}{n-1}$ for all $2\leq i\leq n$.
\epf

Theorem \ref{minSpec} shows that any other graph achieving minimal spectral radius would be $\NDL$-cospectral to the complete graph $K_n$. The next conjecture would follow if it was shown that $K_n$ has no $\NDL$-cospectral mates. Using {\em Sage} \cite{Sage1}, we can verify that $K_n$ is the only graph achieving minimum $\rho_{\NDL}$ for $n\leq 20$.

\begin{conj}
For a graph on $n$ vertices, $$\rho_{\NDL}=\partial^{\cL}_n= \frac{n}{n-1}$$ if and only if $G$ is the complete graph $K_n$, and so $K_n$ is the only graph achieving minimum spectral radius with respect to $\NDL$.
\end{conj}

We may also bound the eigenvalues of one matrix in terms of the other. Butler described a relationship between the eigenvalues of $L$ and $\cL$. The next result appears with proof in \cite{But08} and is stated without proof for weighted graphs in \cite{BC}; one can verify the proof from \cite{But08} remains valid for weighted graphs.

 \begin{thm}{\em \cite[Theorem 4]{But08}}\label{maxdegBut}
Let $G$ be a weighted graph with $\Delta$ the maximum degree of a vertex in $G$ and $\delta$ the minimum degree of a vertex in $G$. Then for $1\leq i\leq n$,
\[\frac{1}{\Delta}\phi_i\leq \mu_i \leq \frac{1}{\delta}\phi_i \]
 \end{thm}

Since the above result holds for weighted graphs, we can again apply the result to the weighted complete graph $W(G)$ to obtain a similar result for $\NDL$. Note the distance Laplacian of $G$ is the combinatorial Laplacian of the weighted complete graph $W(G)$, so $\phi_i(W(G))=\partial^L_i(G).$

 \begin{cor}\label{MaxMinTransBds}
Let $G$ be a graph with $\tM$ the maximum transmission of a vertex in $G$ and $\tm$ the minimum transmission of a vertex in $G$. Then for $1\leq i\leq n$,
\[\frac{1}{\tM}\partial^{L}_i\leq \partial^{\cL}_i \leq \frac{1}{\tm}\partial^{L}_i \]
 \end{cor}

\section{Using twin vertices to determine eigenvalues}\label{sec:EigenDet}

A pair of vertices $u$ and $v$ in $G$ are called {\em twins} if they have the same neighborhood, and the same edge weights in the case of a weighted graph. If $uv$ is an edge in $G$, they are called {\em adjacent twins} and if $uv$ is not an edge in $G$, they are called {\em non-adjacent twins}. Twins have proved very useful in the study of spectra. In this section, we show how twin vertices can be used to compute eigenvalues of $\NDL$ and apply these results to compute the spectra for several families of matrices.

\begin{thm}\label{twins}{\em \cite{BC}}
If a weighted graph $G$ has a set of two or more nonadjacent twins, then 1 is an eigenvalue of $\cL(G)$ and 0 is an eigenvalue of $A(G)$. If a weighted graph $G$ has a set of two or more adjacent twins of degree $d$, then $\frac{d+1}{d}$ is an eigenvalue of $\cL(G)$ and $-1$ is an eigenvalue of $A(G)$.
\end{thm}

Applying part of this result to $W(G)$ gives the analogous result for adjacent twins. In the weighted complete graph $W(G)$, $v_1$ and $v_2$ are adjacent twins with degree $k$. Then by Theorem \ref{twins}, $\frac{k+1}{k}$ is an eigenvalue of $\cL(W(G))$ and thus of $\NDL(G)$.

\begin{cor}\label{adjtwins}
Let $G$ be a graph with $v_1,v_2\in V(G)$ such that $v_1$ and $v_2$ are adjacent twins and $\tr(v_1)=\tr(v_2)=k$. Then $\frac{k+1}{k}$ is an eigenvalue of $\NDL(G)$. 
\end{cor}



Theorem \ref{twins} cannot be used to prove anything for non-adjacent twins, since all vertices are adjacent in the weighted complete graph. However, the proof of the following result adapts the method used to prove Theorem \ref{twins} to the normalized distance Laplacian. 

\begin{thm}\label{nonadjtwins}
Let $G$ be a graph with $v_1,v_2\in V(G)$ such that $v_1$ and $v_2$ are non-adjacent twins and $\tr(v_1)=\tr(v_2)=k$. Then $\frac{k+2}{k}$ is an eigenvalue of $\NDL(G)$ with eigenvector $\mathbf{x}=[1,-1,0,\dots,0]^T$.
\end{thm}

\bpf
Observe for $i=3,\dots,n$, $\NDL_{1,i}=\NDL_{i,1}=\NDL_{2,i}=\NDL_{i,2}=-\frac{d(v_1,v_i)}{\sqrt{kt(v_i)}}=-\frac{d(v_2,v_i)}{\sqrt{kt(v_i)}}$ so the first and second rows and the first and second columns are the same except for in the $2\times 2$ submatrix indexed by $v_1, v_2$. This submatrix is $\bmt 1& -\frac{2}{k}\\
-\frac{2}{k}&1 \emt.$ Multiplying $\NDL$ by $\mathbf{x}$ gives
\begin{eqnarray*}
    \NDL(G)\mathbf{x}=\bmt 1+\frac{2}{k}\\ -\frac{2}{k}-1 \\ \NDL_{3,1}-\NDL_{3,2}\\ \vdots \\ \NDL_{n,1}-\NDL_{n,2}  \emt=\bmt \frac{k+2}{k}\\ -\frac{k+2}{k} \\ 0\\ \vdots \\ 0  \emt=\frac{k+2}{k}\mathbf{x}.
\end{eqnarray*}
So we see $\frac{k+2}{k}$ is an eigenvalue of $\NDL(G)$ with eigenvector $\mathbf{x}$, as desired. 
\epf

We can now apply Corollary \ref{adjtwins} and Theorem \ref{nonadjtwins} to compute the $\NDL$-spectrum of some well known families.

\begin{thm}
The complete bipartite graph on $n+m$ vertices $K_{m,n}$ has \[\spec(\NDL)=\{0, \left(\frac{2n+m}{2n+m-2}\right)^{(n-1)}, \left(\frac{n+2m}{n+2m-2}\right)^{(m-1)}, \frac{2(n^2+m^2+mn-n-m)}{(2n+m-2)(n+2m-2)}\}\].
\end{thm}
\bpf
Let $A=\{a_1,\dots,a_n\}$ and $B=\{b_1,\dots,b_m\}$ be the partite sets of $K_{n,m}$. Observe every pair of vertices in $A$ are non-adjacent twins with transmission $2(n-1)+m=2n+m-2$. By Theorem \ref{nonadjtwins}, every pair $a_1,a_j$ yields the eigenvalue $\frac{2n+m}{2n+m-2}$ and there are $n-1$ such pairs. Similarly, every pair of vertices in $B$ are non-adjacent twins with transmission $n+2(m-1)=n+2m-2$. By Theorem \ref{nonadjtwins}, every pair $b_1,b_j$ yields the eigenvalue $\frac{n+2m}{n+2m-2}$ and there are $m-1$ such pairs. We also have that 0 is an eigenvalue. We have accounted for $n+m-1$ eigenvalues so only one eigenvalue remains, denote this eigenvalue $\nu$. Observe $\sum\limits_{i=1}^{n+m} \partial^{\cL}_i=\trace(\NDL)=n+m$ so 
\[n+m=0+\frac{2n+m}{2n+m-2}(n-1)+\frac{n+2m}{n+2m-2}(m-1)+\nu. \]

By computation, \[\nu=\frac{2(n^2+m^2+mn-n-m)}{(2n+m-2)(n+2m-2)}.\]

\epf


\begin{cor}
The star graph on $n$ vertices $S_n$ has $\spec(\NDL)=\{0,\frac{2n-2}{2n-3}, \left(\frac{2n-1}{2n-3}\right)^{(n-2)}\}$.
\end{cor}

\begin{thm}
The complete graph on $n$ vertices with one edge removed, $K_n-e$, has $\spec(\NDL)=\{0, \frac{n^2-n+2}{n(n-1)}, \left(\frac{n}{n-1}\right)^{(n-3)}, \frac{n+2}{n}\}.$
\end{thm}

\bpf
Let $V(K_n)=V(K_n-e)=\{v_1,\dots,v_n\}$. By vertex transitivity of the complete graph, let $e=v_1v_2$. Then it is easy to observe $v_1$ and $v_2$ are non-adjacent twins while $v_3,\dots,v_n$ are adjacent twins. Since $\tr(v_1)=\tr(v_2)=n$ and $\tr(v_i)=n-1$ for $i=3,\dots,n$, Theorem \ref{nonadjtwins} shows $\frac{n+2}{n}$ is an eigenvalue with multiplicity 1 while Corollary \ref{adjtwins} shows $\frac{n}{n-1}$ is an eigenvalue with multiplicity $n-3$. Since 0 is always an eigenvalue, we have accounted for $n-3+1+1=n-1$ eigenvalues. Calculation shows the remaining eigenvalue is 
\[
\nu = \frac{n^2-n+2}{n(n-1)}.
 \]

\epf


\section{Characteristic polynomials}\label{sec:GenDistChar}

An alternative to direct computation for determining eigenvalues of a matrix is to compute the characteristic polynomial of the matrix. In this section we define the distance generalized characteristic polynomial and show if the polynomial is equal for two non-isomorphic graphs then the graphs must have the same transmission sequence. Then we generalize a method of computing the $\cL$ characteristic polynomial to the $\NDL$ characteristic polynomial.


Let $N(\lambda,r,G)=\lambda I_n -A(G)+rD(G)$. The {\em generalized characteristic polynomial} is $\phi(\lambda,r,G)=\det(N(\lambda,r,G))$. When the parameters are clear, we also use the notation just $N(G)$ and $\phi(G)$ or even just $N$ and $\phi$ if the graph is also clear. It is known that if $\phi(G)=\phi(H)$ then $G$ and $H$ are $A,L,Q,$ and $\cL$ cospectral. We now define an analogous polynomial for the distance matrices $\D, \DL, \DQ,$ and $\NDL$.

\begin{defn}
Let $N^{\D}(\lambda,r,G)=\lambda I_n -\D(G)+rT(G)$. The {\em distance generalized characteristic polynomial} is $\phi^{\D}(\lambda,r,G)=\det(N^{\D}(\lambda,r,G))$.
\end{defn}

When the parameters intended are clear, we also use the notation just $N^{\D}(G)$ and $\phi^{\D}(G)$ or even just $N^{\D}$ and $\phi^{\D}$ if graph is also clear. 

\begin{thm}\label{thm:CharPolyRecover}
From $\phi^{\D}(\lambda,r,G)$ we can recover the characteristic polynomials for $\D$ and $\DQ$, characteristic polynomial of $\DQ$ up to sign and the characteristic polynomial of $\NDL$ up to a constant.
\end{thm}
\begin{proof}
We show that through proper choices of $\lambda$ and $r$, we can obtain the desired polynomials. First, for $\D$, choose $\lambda=x$ and $r=0$ and observe \[\phi^{\D}(x,0,G)=\det(x I_n -\D(G))=p_{\D(G)}(x).\] For $\DL$, choose $\lambda=-x$ and $r=1$, which gives \[ \phi^{\D}(-x,1,G)=\det(-x I_n -\D(G)+T(G))=(-1)^n p_{\DL(G)}(x).\]If
For $\DQ$, choose $\lambda=x$ and $r=-1$, resulting in
\[\phi^{\D}(x,-1,G)=\det(x I_n -\D(G)-T(G))= p_{\DQ(G)}(x).\]

Finally, for $\NDL$ we choose $\lambda=0$ and $r=-x+1$.
\begin{eqnarray*}
\phi^{\D}(0,-x+1,G)&=&\det(-\D(G)+(-x+1)T(G))\\
&=&\det(T(G))\det(T(G)^{-1/2}(\DL(G)-xT(G))T(G)^{-1/2})\\
&=&\det(T(G))\det(\NDL(G)-xI_n))\\
&=&(-1)^n\det(T(G)) p_{\NDL(G)}(x).
\end{eqnarray*}\end{proof}

\begin{cor}
If $\phi^{\D}(G)=\phi^{\D}(H)$ for two graphs $G$ and $H$, then $G$ and $H$ are $\D$, $\DL$, $\DQ$, and $\NDL$ cospectral.
\end{cor}

\begin{proof}
In Theorem \ref{thm:CharPolyRecover}, we showed the characteristic polynomials of $\D$ and $\DQ$ can be recovered exactly, so it is clear that $G$ and $H$ are $\D$ and $\DQ$ cospectral. Let $G$ have order $n$ and $H$ have order $n'$, then if $\phi^{\D}(G)=\phi^{\D}(H)$, necessarily $n=n'$ since otherwise the polynomials would not have the same degree. Therefore $(-1)^n p_{\DL(G)}(x)=(-1)^{n'} p_{\DL(H)}(x)$ implies $p_{\DL(G)}(x)=p_{\DL(H)}(x)$ so $G$ and $H$ are $\DL$ cospectral.
The leading term for all graphs $G$ of $p_{\NDL(G)}(x)$ is $x^n$, so for some constants $c_1,c_2$,  $c_1p_{\NDL(G)}(x)=c_2p_{\NDL(H)}(x)$ implies $c_1=c_2$ and $p_{\NDL(G)}(x)=p_{\NDL(H)}(x)$. Therefore $G$ and $H$ are $\NDL$ cospectral.
\end{proof}

In \cite{LLWX11}, the authors explore properties of non-isomorphic graphs $G$ and $H$ for which $\phi(G)=\phi(H)$. The next theorem is one of their main results.

\begin{thm} \textnormal{\cite[Theorem 2.1]{LLWX11}}\label{thm:DegSeq}
If $\phi(G)=\phi(H)$, then graphs $G$ and $H$ have the same degree sequence.
\end{thm}


 The proof of the above theorem uses \cite[Lemma 2.3]{LLWX11}, which holds for any diagonal matrix but is applied to the degree matrix. It also uses \cite[Lemma 2.4]{LLWX11}, which is stated specifically for the adjacency matrix. However, this lemma holds for all real symmetric matrices; we state the lemma in its full generality next from its original source.


\begin{lem}\textnormal{\cite[p. 186]{GR01}}\label{lem:eigendecomp} Let $\{\mathbf{y_i}\}$ be a set of orthonormal eigenvectors of the real symmetric matrix $M$ with associated eigenvalues $\lambda_i$ $(i=1,2,...,n)$. Then $(\lambda I_n-M)^{-1}=\sum_{i=1}^n \frac{\mathbf{y_i}\mathbf{y_i}^T}{\lambda-\lambda_i}$.
\end{lem}

We now observe the proof given for \cite[Theorem 2.1]{LLWX11} can be used to show the following more general result.

\begin{thm}\label{thm:GenPoly}
Let $M_1$ and $M_2$ be $n\times n$ real symmetric matrices and let $D_1$ and $D_2$ be $n\times n$ diagonal matrices. If $\det\left(\lambda I_n-M_1+rD_1\right)=\det\left(\lambda I_n-M_2+rD_2\right)$, then $\spec(M_1)=\spec(M_2)$ and $\spec(D_1)=\spec(D_2)$.
\end{thm}

\begin{proof}
That $\spec(M_1)=\spec(M_2)$ is immediate by letting $r=0$. The proof that the degree sequences are the same (Theorem \ref{thm:DegSeq}) is by showing $D(G)$ and $D(H)$ are similar matrices, which here shows $\spec(D_1)=\spec(D_2)$.
\end{proof}

Applying Theorem \ref{thm:GenPoly} to the real symmetric matrices $\D(G)$ and $\D(H)$ and the diagonal matrices $T(G)$ and $T(H)$, we obtain a result for $\phi^{\D}$ as a corollary.

\begin{cor}
If $\phi^{\D}(G)=\phi^{\D}(H)$, then graphs $G$ and $H$ have the same transmission sequence.
\end{cor}

Characteristic polynomials are often difficult to calculate. Because of this, many reduction formulas exist. In \cite{O13}, Osborne provides one such reduction algorithm for the generalized characteristic polynomial $\phi(G)$. For a matrix $M(G)$ and a subset of vertices $\alpha\subset V$, let $M_{\alpha}(G)$ be the submatrix obtained by deleting the rows and columns corresponding to the vertices in $\alpha$ from $M$.

\begin{thm}{\em \cite{O13}}
Let $u$ be a vertex in $G$, let $\Cy(u)$ be the collection of cycles in $G$ containing $u$. Then
\[ \phi (\lambda,r,G)= (\lambda + \deg(u)r) \det(N_u(G)) - \sum_{w\sim u} \det(N_{\{u,w\}} (G))-2\sum_{Z\in \Cy(u)} \det(N_{Z}(G)). \]
\end{thm}

We now prove a reduction result for $\phi^{\D} (G)$ using similar proof techniques.

\begin{thm}
Let $u$ be a vertex in $G$, let $\CP(u)$ denote the cyclic permutations of $S_n$ that do not fix $u$, and let $V(\sigma)$ denote vertices not fixed by a permutation $\sigma$. Then,
\[ \phi^{\D} (\lambda,r,G)= (\lambda + t(u)r) \det(N^{\D}_u(G)) - \sum_{\substack{\sigma \in \CP(u)\\|\sigma|=k}} d(u,\sigma(u))d(\sigma(u),\sigma^2(u))\dots d(\sigma^{k-1}(u),u)  \det(N_{V(\sigma)} (G)). \]
\end{thm}

\bpf
Let the vertices of $G$ be $1=u, 2,\dots,n$ and let $(N(G))_{ij}=n_{ij}$. It is clear
\[n_{ij}=\begin{cases} \lambda + t(i)r & i=j\\ -d(i,j) & \text{else} \end{cases} \]
and $\phi^{\D}(G)=\sum_{\sigma\in S_n} \sgn(\sigma) \prod _{i=1}^{n} n_{i\sigma(i)}$. Partition $S_n$ in to $P_1$ and $P_2$ such that $\sigma\in P_1$ if $\sigma(1)=1$ and otherwise $\sigma\in P_2$. Write $\sigma$ as a product of cycles $\sigma=\sigma_1\sigma_2\dots\sigma_{\ell}$ such that $1\in \sigma_1$. Clearly, \[\phi^{\D}(G)=\sum_{\sigma\in P_1} \sgn(\sigma) \prod _{i=1}^{n} n_{i\sigma(i)}+\sum_{\sigma\in P_2} \sgn(\sigma) \prod _{i=1}^{n} n_{i\sigma(i)}\]  
If $\sigma\in P_1$, then $\sgn(\sigma)=\sgn(\sigma_2\dots\sigma_{\ell})$ and
\begin{eqnarray*}
\sum_{\sigma\in P_1} \sgn(\sigma) \prod _{i=1}^{n} n_{i\sigma(i)} &=& \sum_{\sigma\in P_1} (\lambda + t(1)r) \sgn(\sigma_2\dots\sigma_{\ell})\prod _{i=2}^{n} n_{i\sigma(i)}\\
&=&  (\lambda + t(1)r) \sum_{\sigma\in S_{n-1}} \sgn(\sigma)\prod _{i=2}^{n} n_{i\sigma(i)}\\
&=&  (\lambda + t(1)r) \det(N^{\D}_1(G)).
\end{eqnarray*}
If $\sigma\in P_2$, let $\sigma_1=(1,\sigma(1),\dots,\sigma^{k-1}(1))$. Then
\begin{eqnarray*}
&&\sum_{\sigma\in P_2} \sgn(\sigma) \prod _{i=1}^{n} n_{i\sigma(i)}\\ 
&=&\sum_{\sigma\in P_2} \sgn(\sigma_1) (-1)^k d(1,\sigma(1))d(\sigma(1),\sigma^2(1))\dots d(\sigma^{k-1}(1),1) \sgn(\sigma_2\dots\sigma_k)\prod _{i\in V(\sigma_2\dots\sigma_k)} n_{i\sigma(i)}\\
&=&\sum_{\sigma\in P_2} (-1)^{k-1} (-1)^k d(1,\sigma(1))d(\sigma(1),\sigma^2(1))\dots d(\sigma^{k-1}(1),1) \sgn(\sigma_2\dots\sigma_k)\prod _{i\in V(\sigma_2\dots\sigma_k)} n_{i\sigma(i)}.
\end{eqnarray*}
Fixing $\sigma_1$, consider all other permutations of the remaining vertices. This gives
\begin{eqnarray*}
&&-\sum_{\substack{\sigma\in \CP(1)\\|\sigma|=k}} d(1,\sigma(1))d(\sigma(1),\sigma^2(1))\dots d(\sigma^{k-1}(1),1) \sum_{\tau\in S_{n-|V(\sigma)|}}\sgn(\tau)\prod _{i\in V(\tau)} n_{i\tau(i)}\\
&=&-\sum_{\substack{\sigma\in \CP(1)\\|\sigma|=k}} d(1,\sigma(1))d(\sigma(1),\sigma^2(1))\dots d(\sigma^{k-1}(1),1) \det(N_{V(\sigma)} (G)).
\end{eqnarray*}
\epf

We now shift our focus back to the standard characteristic polynomial. Methods of computing characteristic polynomials have been found for various matrices associated with graphs. Such a method was found for the weighted normalized Laplacian in $\cite{BH16}$ and is given below. A decomposition $D$ of an undirected weighted graph $G$ is a subgraph consisting of disjoint edges and cycles. Let $s(D)$ denote the number of cycles of length at least three in $D$, let $e(D)$ denote the number of cycles in $D$ that have an even number of vertices (here, consider an edge to be a cycle of length two), and let $F(D)$ be the set of isolated edges in the decomposition. Note a decomposition need not be spanning, and in fact the empty decomposition is included. 

\begin{thm}\label{charpoly}{\em \cite{BH16}}
Let $G$ be a weighted graph on $n$ vertices. Then the characteristic polynomial of the normalized Laplacian matrix is
\[p(x)=\sum_{D} (-1)^{e(D)} 2^{s(D)} \frac{\prod _{v_iv_j\in E(D)} w(v_i,v_j) \prod_{v_iv_j\in F(D)} w(v_i,v_j)}{\prod_{v_i\in V(D)} \deg(v_i)} (x-1)^{n-|V(D)|}\]
where the sum runs over all decompositions $D$ of the graph $G$. 
\end{thm}

Applying this result to the normalized distance Laplacian viewed as a weighted normalized Laplacian, we obtain the following formula for computing its characteristic formula. Note $w_{W(G)}(v_i,v_j)=d_G(v_i,v_j)$ and $\deg_{W(G)}(v_i)=\tr_G(v_i)$.
\begin{cor}
Let $G$ be a graph on $n$ vertices. Then the characteristic polynomial of the normalized distance Laplacian matrix is
\[p(x)=\sum_{D} (-1)^{e(D)} 2^{s(D)} (x-1)^{n-|V(D)|} \frac{\prod _{v_iv_j\in E(D)} d_{G}(v_i,v_j) \prod_{v_iv_j\in F(D)} d_G(v_i,v_j)}{\prod_{v_i\in V(D)} \tr_G(v_i)} (x-1)^{n-|V(D)|}\]
where the sum runs over all decompositions $D$ of the complete graph $K_n$. 
\end{cor}

\section{Cospectral graphs}\label{sec:Cospec}
In this section, we show cospectral graphs with respect to the normalized distance Laplacian are rare. In Section \ref{sec:CospecPairs} we exhibit and discuss the 5 cospectral pairs on 8 and 9 vertices as well as interesting examples of cospectral pairs on 10 vertices. We show the number of edges in a graph, degree sequence, transmission sequence, girth, Weiner index, planarity, $k$-regularity, and $k$-transmission regularity are not preserved by $\NDL$-cospectrality. We compare $\NDL$-cospectralitity with $\DL$-cospectrality and provide examples of pairs of graphs that are both $\NDL$-cospectral and $\DL$-cospectral, as well as pairs only cospectral with respect to one of the matrices. We also exhibit graphs that are cospectral only for $\NDL$ and graphs that are cospectral for all matrices $A,L,Q,\cL,\D,\DL,\DQ,\NDL$. In Section \ref{sec:cospecCompare}, the number of graphs on 10 or fewer vertices with a $\NDL$-cospectral pair is computed and compared with the number of graphs with a $M$-cospectral pair, where $M=A,L,Q,\cL,\D,\DL,\DQ$. That section also includes a discussion of computational methods.

\subsection{Cospectral pairs on 10 or fewer vertices}\label{sec:CospecPairs}
The first instance of cospectral graphs with respect to the normalized distance Laplacian occurs on 8 vertices and there is only one such pair, shown in Figure \ref{fig:cospec8}. Using {\em Sage} \cite{Sage2} we can compute their $\NDL$ characteristic polynomial: $p_{\NDL}(x)=x^8-8x^7+\frac{317947}{11616}x^6- \frac{5428399}{104544}x^5+\frac{24668087}{418176}x^4-\frac{4196075}{104544}x^3+\frac{575771}{38016}x^2-\frac{85211}{34848}x$. The only difference between the graphs is the {\red light} colored edge, and we refer to the maximal shared subgraph (i.e. the graph that results in removing the {\red light} colored edge from either graph) as the {\em base graph}.
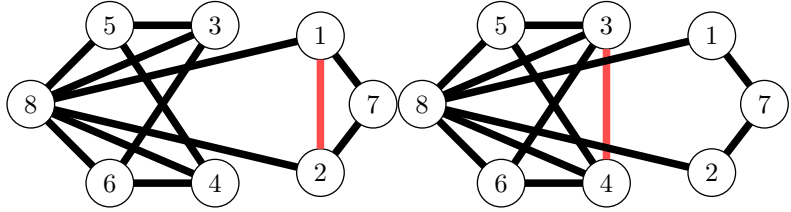
\begin{figure}[h!]
\begin{center}
\begin{tikzpicture}[scale=0.7]
\definecolor{cv0}{rgb}{0.0,0.0,0.0}
\definecolor{cfv0}{rgb}{1.0,1.0,1.0}
\definecolor{clv0}{rgb}{0.0,0.0,0.0}
\definecolor{cv1}{rgb}{0.0,0.0,0.0}
\definecolor{cfv1}{rgb}{1.0,1.0,1.0}
\definecolor{clv1}{rgb}{0.0,0.0,0.0}
\definecolor{cv2}{rgb}{0.0,0.0,0.0}
\definecolor{cfv2}{rgb}{1.0,1.0,1.0}
\definecolor{clv2}{rgb}{0.0,0.0,0.0}
\definecolor{cv3}{rgb}{0.0,0.0,0.0}
\definecolor{cfv3}{rgb}{1.0,1.0,1.0}
\definecolor{clv3}{rgb}{0.0,0.0,0.0}
\definecolor{cv4}{rgb}{0.0,0.0,0.0}
\definecolor{cfv4}{rgb}{1.0,1.0,1.0}
\definecolor{clv4}{rgb}{0.0,0.0,0.0}
\definecolor{cv5}{rgb}{0.0,0.0,0.0}
\definecolor{cfv5}{rgb}{1.0,1.0,1.0}
\definecolor{clv5}{rgb}{0.0,0.0,0.0}
\definecolor{cv6}{rgb}{0.0,0.0,0.0}
\definecolor{cfv6}{rgb}{1.0,1.0,1.0}
\definecolor{clv6}{rgb}{0.0,0.0,0.0}
\definecolor{cv7}{rgb}{0.0,0.0,0.0}
\definecolor{cfv7}{rgb}{1.0,1.0,1.0}
\definecolor{clv7}{rgb}{0.0,0.0,0.0}
\definecolor{cv0v1}{rgb}{0.0,0.0,0.0}
\definecolor{cv0v3}{rgb}{0.0,0.0,0.0}
\definecolor{cv0v6}{rgb}{0.0,0.0,0.0}
\definecolor{cv1v3}{rgb}{0.0,0.0,0.0}
\definecolor{cv1v6}{rgb}{0.0,0.0,0.0}
\definecolor{cv2v3}{rgb}{0.0,0.0,0.0}
\definecolor{cv2v5}{rgb}{0.0,0.0,0.0}
\definecolor{cv2v7}{rgb}{0.0,0.0,0.0}
\definecolor{cv3v4}{rgb}{0.0,0.0,0.0}
\definecolor{cv3v5}{rgb}{0.0,0.0,0.0}
\definecolor{cv3v7}{rgb}{0.0,0.0,0.0}
\definecolor{cv4v5}{rgb}{0.0,0.0,0.0}
\definecolor{cv4v7}{rgb}{0.0,0.0,0.0}
\Vertex[style={minimum
size=1.0cm,draw=cv0,fill=cfv0,text=clv0,shape=circle},LabelOut=false,L=\hbox{$2$},x=4cm,y=.2cm]{v0}
\Vertex[style={minimum
size=1.0cm,draw=cv1,fill=cfv1,text=clv1,shape=circle},LabelOut=false,L=\hbox{$1$},x=4cm,y=2.8cm]{v1}
\Vertex[style={minimum
size=1.0cm,draw=cv2,fill=cfv2,text=clv2,shape=circle},LabelOut=false,L=\hbox{$6$},x=0cm,y=0cm]{v2}
\Vertex[style={minimum
size=1.0cm,draw=cv3,fill=cfv3,text=clv3,shape=circle},LabelOut=false,L=\hbox{$8$},x=-1.5cm,y=1.5cm]{v3}
\Vertex[style={minimum
size=1.0cm,draw=cv4,fill=cfv4,text=clv4,shape=circle},LabelOut=false,L=\hbox{$5$},x=0cm,y=3cm]{v4}
\Vertex[style={minimum
size=1.0cm,draw=cv5,fill=cfv5,text=clv5,shape=circle},LabelOut=false,L=\hbox{$4$},x=2,y=0.0cm]{v5}
\Vertex[style={minimum
size=1.0cm,draw=cv6,fill=cfv6,text=clv6,shape=circle},LabelOut=false,L=\hbox{$7$},x=5cm,y=1.5cm]{v6}
\Vertex[style={minimum
size=1.0cm,draw=cv7,fill=cfv7,text=clv7,shape=circle},LabelOut=false,L=\hbox{$3$},x=2cm,y=3cm]{v7}
\Edge[lw=0.1cm,style={color=red!70,},](v0)(v1)
\Edge[lw=0.1cm,style={color=cv0v3,},](v0)(v3)
\Edge[lw=0.1cm,style={color=cv0v6,},](v0)(v6)
\Edge[lw=0.1cm,style={color=cv1v3,},](v1)(v3)
\Edge[lw=0.1cm,style={color=cv1v6,},](v1)(v6)
\Edge[lw=0.1cm,style={color=cv2v3,},](v2)(v3)
\Edge[lw=0.1cm,style={color=cv2v5,},](v2)(v5)
\Edge[lw=0.1cm,style={color=cv2v7,},](v2)(v7)
\Edge[lw=0.1cm,style={color=cv3v4,},](v3)(v4)
\Edge[lw=0.1cm,style={color=cv3v5,},](v3)(v5)
\Edge[lw=0.1cm,style={color=cv3v7,},](v3)(v7)
\Edge[lw=0.1cm,style={color=cv4v5,},](v4)(v5)
\Edge[lw=0.1cm,style={color=cv4v7,},](v4)(v7)
\end{tikzpicture}
\begin{tikzpicture}[scale=0.7]
\definecolor{cv0}{rgb}{0.0,0.0,0.0}
\definecolor{cfv0}{rgb}{1.0,1.0,1.0}
\definecolor{clv0}{rgb}{0.0,0.0,0.0}
\definecolor{cv1}{rgb}{0.0,0.0,0.0}
\definecolor{cfv1}{rgb}{1.0,1.0,1.0}
\definecolor{clv1}{rgb}{0.0,0.0,0.0}
\definecolor{cv2}{rgb}{0.0,0.0,0.0}
\definecolor{cfv2}{rgb}{1.0,1.0,1.0}
\definecolor{clv2}{rgb}{0.0,0.0,0.0}
\definecolor{cv3}{rgb}{0.0,0.0,0.0}
\definecolor{cfv3}{rgb}{1.0,1.0,1.0}
\definecolor{clv3}{rgb}{0.0,0.0,0.0}
\definecolor{cv4}{rgb}{0.0,0.0,0.0}
\definecolor{cfv4}{rgb}{1.0,1.0,1.0}
\definecolor{clv4}{rgb}{0.0,0.0,0.0}
\definecolor{cv5}{rgb}{0.0,0.0,0.0}
\definecolor{cfv5}{rgb}{1.0,1.0,1.0}
\definecolor{clv5}{rgb}{0.0,0.0,0.0}
\definecolor{cv6}{rgb}{0.0,0.0,0.0}
\definecolor{cfv6}{rgb}{1.0,1.0,1.0}
\definecolor{clv6}{rgb}{0.0,0.0,0.0}
\definecolor{cv7}{rgb}{0.0,0.0,0.0}
\definecolor{cfv7}{rgb}{1.0,1.0,1.0}
\definecolor{clv7}{rgb}{0.0,0.0,0.0}
\definecolor{cv0v1}{rgb}{0.0,0.0,0.0}
\definecolor{cv0v3}{rgb}{0.0,0.0,0.0}
\definecolor{cv0v6}{rgb}{0.0,0.0,0.0}
\definecolor{cv1v3}{rgb}{0.0,0.0,0.0}
\definecolor{cv1v6}{rgb}{0.0,0.0,0.0}
\definecolor{cv2v3}{rgb}{0.0,0.0,0.0}
\definecolor{cv2v5}{rgb}{0.0,0.0,0.0}
\definecolor{cv2v7}{rgb}{0.0,0.0,0.0}
\definecolor{cv3v4}{rgb}{0.0,0.0,0.0}
\definecolor{cv3v5}{rgb}{0.0,0.0,0.0}
\definecolor{cv3v7}{rgb}{0.0,0.0,0.0}
\definecolor{cv4v5}{rgb}{0.0,0.0,0.0}
\definecolor{cv4v7}{rgb}{0.0,0.0,0.0}
\Vertex[style={minimum
size=1.0cm,draw=cv0,fill=cfv0,text=clv0,shape=circle},LabelOut=false,L=\hbox{$2$},x=4cm,y=.2cm]{v0}
\Vertex[style={minimum
size=1.0cm,draw=cv1,fill=cfv1,text=clv1,shape=circle},LabelOut=false,L=\hbox{$1$},x=4cm,y=2.8cm]{v1}
\Vertex[style={minimum
size=1.0cm,draw=cv2,fill=cfv2,text=clv2,shape=circle},LabelOut=false,L=\hbox{$6$},x=0cm,y=0cm]{v2}
\Vertex[style={minimum
size=1.0cm,draw=cv3,fill=cfv3,text=clv3,shape=circle},LabelOut=false,L=\hbox{$8$},x=-1.5cm,y=1.5cm]{v3}
\Vertex[style={minimum
size=1.0cm,draw=cv4,fill=cfv4,text=clv4,shape=circle},LabelOut=false,L=\hbox{$5$},x=0cm,y=3cm]{v4}
\Vertex[style={minimum
size=1.0cm,draw=cv5,fill=cfv5,text=clv5,shape=circle},LabelOut=false,L=\hbox{$4$},x=2,y=0.0cm]{v5}
\Vertex[style={minimum
size=1.0cm,draw=cv6,fill=cfv6,text=clv6,shape=circle},LabelOut=false,L=\hbox{$7$},x=5cm,y=1.5cm]{v6}
\Vertex[style={minimum
size=1.0cm,draw=cv7,fill=cfv7,text=clv7,shape=circle},LabelOut=false,L=\hbox{$3$},x=2cm,y=3cm]{v7}
\Edge[lw=0.1cm,style={color=red!70,},](v7)(v5)
\Edge[lw=0.1cm,style={color=cv0v3,},](v0)(v3)
\Edge[lw=0.1cm,style={color=cv0v6,},](v0)(v6)
\Edge[lw=0.1cm,style={color=cv1v3,},](v1)(v3)
\Edge[lw=0.1cm,style={color=cv1v6,},](v1)(v6)
\Edge[lw=0.1cm,style={color=cv2v3,},](v2)(v3)
\Edge[lw=0.1cm,style={color=cv2v5,},](v2)(v5)
\Edge[lw=0.1cm,style={color=cv2v7,},](v2)(v7)
\Edge[lw=0.1cm,style={color=cv3v4,},](v3)(v4)
\Edge[lw=0.1cm,style={color=cv3v5,},](v3)(v5)
\Edge[lw=0.1cm,style={color=cv3v7,},](v3)(v7)
\Edge[lw=0.1cm,style={color=cv4v5,},](v4)(v5)
\Edge[lw=0.1cm,style={color=cv4v7,},](v4)(v7)
\end{tikzpicture}
\end{center}
\caption{The only $\NDL$-cospectral pair on 8 vertices}
\label{fig:cospec8}
\end{figure}


If a graph $G$ has vertices $v_1,v_2,v_3,$ and $v_4$ such that $v_1$ and $v_2$ are non-adjacent twins, $v_3$ and $v_4$ are non-adjacent twins, and $\tr(v_1)=\tr(v_3)$ then we say $\{\{v_1,v_2\},\{v_3,v_4\}\}$ is a set of {\em co-transmission twins}. In \cite{BDHLRSY19}, a cospectral construction is described for the distance Laplacian using co-transmission twins. If a graph $G$ has co-transmission twins $\{\{v_1,v_2\},\{v_3,v_4\}\}$, then $G+v_1v_2$ and $G+v_3v_4$ are $\DL$-cospectral. Note that the graphs in Figure \ref{fig:cospec8} can be constructed this way from their base graph with co-transmission twins $\{\{1,2\},\{3,4\}\}$, so they are $\DL$-cospectral as well. Their $\DL$ characteristic polynomial is $p_{\DL}(x)=x^8 - 94x^7 + 3756x^6 - 82728x^5 + 1084992x^4 - 8473984x^3 +
36492288x^2 - 66834432x$. However, this construction does not always find $\NDL$-cospectral graphs. The graphs in Figure \ref{fig:notcospec} can be constructed from their base graph using co-transmission twins $\{\{1,2\},\{3,4\}\}$, so they are cospectral with respect to $\DL$, but they are not cospectral with respect to $\NDL$. 


\begin{figure}[ht]
    \centering
\begin{tikzpicture}[scale=0.7]
\definecolor{cv0}{rgb}{0.0,0.0,0.0}
\definecolor{cfv0}{rgb}{1.0,1.0,1.0}
\definecolor{clv0}{rgb}{0.0,0.0,0.0}
\definecolor{cv1}{rgb}{0.0,0.0,0.0}
\definecolor{cfv1}{rgb}{1.0,1.0,1.0}
\definecolor{clv1}{rgb}{0.0,0.0,0.0}
\definecolor{cv2}{rgb}{0.0,0.0,0.0}
\definecolor{cfv2}{rgb}{1.0,1.0,1.0}
\definecolor{clv2}{rgb}{0.0,0.0,0.0}
\definecolor{cv3}{rgb}{0.0,0.0,0.0}
\definecolor{cfv3}{rgb}{1.0,1.0,1.0}
\definecolor{clv3}{rgb}{0.0,0.0,0.0}
\definecolor{cv4}{rgb}{0.0,0.0,0.0}
\definecolor{cfv4}{rgb}{1.0,1.0,1.0}
\definecolor{clv4}{rgb}{0.0,0.0,0.0}
\definecolor{cv5}{rgb}{0.0,0.0,0.0}
\definecolor{cfv5}{rgb}{1.0,1.0,1.0}
\definecolor{clv5}{rgb}{0.0,0.0,0.0}
\definecolor{cv6}{rgb}{0.0,0.0,0.0}
\definecolor{cfv6}{rgb}{1.0,1.0,1.0}
\definecolor{clv6}{rgb}{0.0,0.0,0.0}
\definecolor{cv7}{rgb}{0.0,0.0,0.0}
\definecolor{cfv7}{rgb}{1.0,1.0,1.0}
\definecolor{clv7}{rgb}{0.0,0.0,0.0}
\definecolor{cv0v1}{rgb}{0.0,0.0,0.0}
\definecolor{cv1v2}{rgb}{0.0,0.0,0.0}
\definecolor{cv1v3}{rgb}{0.0,0.0,0.0}
\definecolor{cv2v3}{rgb}{0.0,0.0,0.0}
\definecolor{cv2v4}{rgb}{0.0,0.0,0.0}
\definecolor{cv2v5}{rgb}{0.0,0.0,0.0}
\definecolor{cv3v4}{rgb}{0.0,0.0,0.0}
\definecolor{cv3v5}{rgb}{0.0,0.0,0.0}
\definecolor{cv4v6}{rgb}{0.0,0.0,0.0}
\definecolor{cv4v7}{rgb}{0.0,0.0,0.0}
\definecolor{cv5v6}{rgb}{0.0,0.0,0.0}
\definecolor{cv5v7}{rgb}{0.0,0.0,0.0}
\Vertex[style={minimum
size=1.0cm,draw=cv0,fill=cfv0,text=clv0,shape=circle},LabelOut=false,L=\hbox{$8$},x=7cm,y=1cm]{v0}
\Vertex[style={minimum
size=1.0cm,draw=cv1,fill=cfv1,text=clv1,shape=circle},LabelOut=false,L=\hbox{$7$},x=5.5cm,y=1cm]{v1}
\Vertex[style={minimum
size=1.0cm,draw=cv2,fill=cfv2,text=clv2,shape=circle},LabelOut=false,L=\hbox{$2$},x=4cm,y=2cm]{v2}
\Vertex[style={minimum
size=1.0cm,draw=cv3,fill=cfv3,text=clv3,shape=circle},LabelOut=false,L=\hbox{$1$},x=4cm,y=0cm]{v3}
\Vertex[style={minimum
size=1.0cm,draw=cv4,fill=cfv4,text=clv4,shape=circle},LabelOut=false,L=\hbox{$4$},x=2cm,y=2cm]{v4}
\Vertex[style={minimum
size=1.0cm,draw=cv5,fill=cfv5,text=clv5,shape=circle},LabelOut=false,L=\hbox{$3$},x=2cm,y=0cm]{v5}
\Vertex[style={minimum
size=1.0cm,draw=cv6,fill=cfv6,text=clv6,shape=circle},LabelOut=false,L=\hbox{$6$},x=0.0cm,y=2cm]{v6}
\Vertex[style={minimum
size=1.0cm,draw=cv7,fill=cfv7,text=clv7,shape=circle},LabelOut=false,L=\hbox{$5$},x=0cm,y=0cm]{v7}
\Edge[lw=0.1cm,style={color=cv0v1,},](v0)(v1)
\Edge[lw=0.1cm,style={color=cv1v2,},](v1)(v2)
\Edge[lw=0.1cm,style={color=cv1v3,},](v1)(v3)
\Edge[lw=0.1cm,style={color=red!70,},](v2)(v3)
\Edge[lw=0.1cm,style={color=cv2v4,},](v2)(v4)
\Edge[lw=0.1cm,style={color=cv2v5,},](v2)(v5)
\Edge[lw=0.1cm,style={color=cv3v4,},](v3)(v4)
\Edge[lw=0.1cm,style={color=cv3v5,},](v3)(v5)
\Edge[lw=0.1cm,style={color=cv4v6,},](v4)(v6)
\Edge[lw=0.1cm,style={color=cv4v7,},](v4)(v7)
\Edge[lw=0.1cm,style={color=cv5v6,},](v5)(v6)
\Edge[lw=0.1cm,style={color=cv5v7,},](v5)(v7)
\end{tikzpicture}
\begin{tikzpicture}[scale=0.7]
\definecolor{cv0}{rgb}{0.0,0.0,0.0}
\definecolor{cfv0}{rgb}{1.0,1.0,1.0}
\definecolor{clv0}{rgb}{0.0,0.0,0.0}
\definecolor{cv1}{rgb}{0.0,0.0,0.0}
\definecolor{cfv1}{rgb}{1.0,1.0,1.0}
\definecolor{clv1}{rgb}{0.0,0.0,0.0}
\definecolor{cv2}{rgb}{0.0,0.0,0.0}
\definecolor{cfv2}{rgb}{1.0,1.0,1.0}
\definecolor{clv2}{rgb}{0.0,0.0,0.0}
\definecolor{cv3}{rgb}{0.0,0.0,0.0}
\definecolor{cfv3}{rgb}{1.0,1.0,1.0}
\definecolor{clv3}{rgb}{0.0,0.0,0.0}
\definecolor{cv4}{rgb}{0.0,0.0,0.0}
\definecolor{cfv4}{rgb}{1.0,1.0,1.0}
\definecolor{clv4}{rgb}{0.0,0.0,0.0}
\definecolor{cv5}{rgb}{0.0,0.0,0.0}
\definecolor{cfv5}{rgb}{1.0,1.0,1.0}
\definecolor{clv5}{rgb}{0.0,0.0,0.0}
\definecolor{cv6}{rgb}{0.0,0.0,0.0}
\definecolor{cfv6}{rgb}{1.0,1.0,1.0}
\definecolor{clv6}{rgb}{0.0,0.0,0.0}
\definecolor{cv7}{rgb}{0.0,0.0,0.0}
\definecolor{cfv7}{rgb}{1.0,1.0,1.0}
\definecolor{clv7}{rgb}{0.0,0.0,0.0}
\definecolor{cv0v1}{rgb}{0.0,0.0,0.0}
\definecolor{cv1v2}{rgb}{0.0,0.0,0.0}
\definecolor{cv1v3}{rgb}{0.0,0.0,0.0}
\definecolor{cv2v3}{rgb}{0.0,0.0,0.0}
\definecolor{cv2v4}{rgb}{0.0,0.0,0.0}
\definecolor{cv2v5}{rgb}{0.0,0.0,0.0}
\definecolor{cv3v4}{rgb}{0.0,0.0,0.0}
\definecolor{cv3v5}{rgb}{0.0,0.0,0.0}
\definecolor{cv4v6}{rgb}{0.0,0.0,0.0}
\definecolor{cv4v7}{rgb}{0.0,0.0,0.0}
\definecolor{cv5v6}{rgb}{0.0,0.0,0.0}
\definecolor{cv5v7}{rgb}{0.0,0.0,0.0}
\Vertex[style={minimum
size=1.0cm,draw=cv0,fill=cfv0,text=clv0,shape=circle},LabelOut=false,L=\hbox{$8$},x=7cm,y=1cm]{v0}
\Vertex[style={minimum
size=1.0cm,draw=cv1,fill=cfv1,text=clv1,shape=circle},LabelOut=false,L=\hbox{$7$},x=5.5cm,y=1cm]{v1}
\Vertex[style={minimum
size=1.0cm,draw=cv2,fill=cfv2,text=clv2,shape=circle},LabelOut=false,L=\hbox{$2$},x=4cm,y=2cm]{v2}
\Vertex[style={minimum
size=1.0cm,draw=cv3,fill=cfv3,text=clv3,shape=circle},LabelOut=false,L=\hbox{$1$},x=4cm,y=0cm]{v3}
\Vertex[style={minimum
size=1.0cm,draw=cv4,fill=cfv4,text=clv4,shape=circle},LabelOut=false,L=\hbox{$4$},x=2cm,y=2cm]{v4}
\Vertex[style={minimum
size=1.0cm,draw=cv5,fill=cfv5,text=clv5,shape=circle},LabelOut=false,L=\hbox{$3$},x=2cm,y=0cm]{v5}
\Vertex[style={minimum
size=1.0cm,draw=cv6,fill=cfv6,text=clv6,shape=circle},LabelOut=false,L=\hbox{$6$},x=0.0cm,y=2cm]{v6}
\Vertex[style={minimum
size=1.0cm,draw=cv7,fill=cfv7,text=clv7,shape=circle},LabelOut=false,L=\hbox{$5$},x=0cm,y=0cm]{v7}
\Edge[lw=0.1cm,style={color=cv0v1,},](v0)(v1)
\Edge[lw=0.1cm,style={color=cv1v2,},](v1)(v2)
\Edge[lw=0.1cm,style={color=cv1v3,},](v1)(v3)
\Edge[lw=0.1cm,style={color=red!70,},](v4)(v5)
\Edge[lw=0.1cm,style={color=cv2v4,},](v2)(v4)
\Edge[lw=0.1cm,style={color=cv2v5,},](v2)(v5)
\Edge[lw=0.1cm,style={color=cv3v4,},](v3)(v4)
\Edge[lw=0.1cm,style={color=cv3v5,},](v3)(v5)
\Edge[lw=0.1cm,style={color=cv4v6,},](v4)(v6)
\Edge[lw=0.1cm,style={color=cv4v7,},](v4)(v7)
\Edge[lw=0.1cm,style={color=cv5v6,},](v5)(v6)
\Edge[lw=0.1cm,style={color=cv5v7,},](v5)(v7)
\end{tikzpicture}
    \caption{Graphs that are $\DL$ but not $\NDL$-cospectral using the co-transmission twins construction}
    \label{fig:notcospec}
\end{figure}
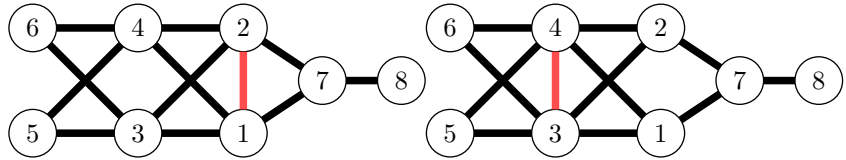

There are only four pairs of $\NDL$-cospectral graphs on 9 vertices. Three of the pairs that are $\NDL$-cospectral differ by only one edge and have related base graphs. In Figure \ref{fig:3pairscospec9}, the three pairs can be seen by including 0, 1, or 2 of the dashed edges $\{1,4\}$ and $\{2,3\}$ (note including just the edge $\{1,4\}$ or just the edge $\{2,3\}$ creates isomorphic graphs, so we only need consider one of these cases). When $0$ dashed edges are included, their $\NDL$ characteristic polynomial is $p_{\NDL}(x)=x^9 - 9x^8 + \frac{1926013}{54450}x^7 - \frac{259072321}{3267000}x^6 +
\frac{2717888893}{24502500}x^5 -\frac{233194363}{2352240}x^4 +
\frac{243851297233}{4410450000}x^3 - \frac{587831111}{33412500}x^2 +
\frac{674126228}{275653125}x$. When 1 dashed edge is included, their $\NDL$ characteristic polynomial is $p_{\NDL}(x)=x^9 - 9x^8 + \frac{1926211}{54450}x^7 - \frac{9598153}{121000}x^6 +
\frac{1812984073}{16335000}x^5 - \frac{24314025553}{245025000}x^4 +\frac{67829453381}{1225125000}x^3 - \frac{10796929657}{612562500}x^2 + \frac{758404}{309375}x$. When both dashed edges are included, their $\NDL$ characteristic polynomial is $p_{\NDL}(x)=x^9 - 9x^8 + \frac{3852769}{108900}x^7 - \frac{21601501}{272250}x^6 +\frac{27208546}{245025}x^5 -\frac{6083465273}{61256250}x^4 +\frac{11318237801}{204187500}x^3 - \frac{1228870232}{69609375}x^2 + \frac{15544256}{6328125}x$. 

\begin{figure}[ht]
    \centering
\begin{tikzpicture}[scale=0.7]
\definecolor{cv0}{rgb}{0.0,0.0,0.0}
\definecolor{cfv0}{rgb}{1.0,1.0,1.0}
\definecolor{clv0}{rgb}{0.0,0.0,0.0}
\definecolor{cv1}{rgb}{0.0,0.0,0.0}
\definecolor{cfv1}{rgb}{1.0,1.0,1.0}
\definecolor{clv1}{rgb}{0.0,0.0,0.0}
\definecolor{cv2}{rgb}{0.0,0.0,0.0}
\definecolor{cfv2}{rgb}{1.0,1.0,1.0}
\definecolor{clv2}{rgb}{0.0,0.0,0.0}
\definecolor{cv3}{rgb}{0.0,0.0,0.0}
\definecolor{cfv3}{rgb}{1.0,1.0,1.0}
\definecolor{clv3}{rgb}{0.0,0.0,0.0}
\definecolor{cv4}{rgb}{0.0,0.0,0.0}
\definecolor{cfv4}{rgb}{1.0,1.0,1.0}
\definecolor{clv4}{rgb}{0.0,0.0,0.0}
\definecolor{cv5}{rgb}{0.0,0.0,0.0}
\definecolor{cfv5}{rgb}{1.0,1.0,1.0}
\definecolor{clv5}{rgb}{0.0,0.0,0.0}
\definecolor{cv6}{rgb}{0.0,0.0,0.0}
\definecolor{cfv6}{rgb}{1.0,1.0,1.0}
\definecolor{clv6}{rgb}{0.0,0.0,0.0}
\definecolor{cv7}{rgb}{0.0,0.0,0.0}
\definecolor{cfv7}{rgb}{1.0,1.0,1.0}
\definecolor{clv7}{rgb}{0.0,0.0,0.0}
\definecolor{cv8}{rgb}{0.0,0.0,0.0}
\definecolor{cfv8}{rgb}{1.0,1.0,1.0}
\definecolor{clv8}{rgb}{0.0,0.0,0.0}
\definecolor{cv0v3}{rgb}{0.0,0.0,0.0}
\definecolor{cv0v4}{rgb}{0.0,0.0,0.0}
\definecolor{cv0v7}{rgb}{0.0,0.0,0.0}
\definecolor{cv0v8}{rgb}{0.0,0.0,0.0}
\definecolor{cv1v2}{rgb}{0.0,0.0,0.0}
\definecolor{cv1v4}{rgb}{0.0,0.0,0.0}
\definecolor{cv1v7}{rgb}{0.0,0.0,0.0}
\definecolor{cv1v8}{rgb}{0.0,0.0,0.0}
\definecolor{cv2v4}{rgb}{0.0,0.0,0.0}
\definecolor{cv2v5}{rgb}{0.0,0.0,0.0}
\definecolor{cv2v6}{rgb}{0.0,0.0,0.0}
\definecolor{cv3v4}{rgb}{0.0,0.0,0.0}
\definecolor{cv3v5}{rgb}{0.0,0.0,0.0}
\definecolor{cv3v6}{rgb}{0.0,0.0,0.0}
\definecolor{cv4v5}{rgb}{0.0,0.0,0.0}
\definecolor{cv4v6}{rgb}{0.0,0.0,0.0}
\definecolor{cv7v8}{rgb}{0.0,0.0,0.0}
\Vertex[style={minimum
size=1.0cm,draw=cv0,fill=cfv0,text=clv0,shape=circle},LabelOut=false,L=\hbox{$2$},x=5cm,y=0cm]{v0}
\Vertex[style={minimum
size=1.0cm,draw=cv1,fill=cfv1,text=clv1,shape=circle},LabelOut=false,L=\hbox{$1$},x=5cm,y=3cm]{v1}
\Vertex[style={minimum
size=1.0cm,draw=cv2,fill=cfv2,text=clv2,shape=circle},LabelOut=false,L=\hbox{$3$},x=3cm,y=3cm]{v2}
\Vertex[style={minimum
size=1.0cm,draw=cv3,fill=cfv3,text=clv3,shape=circle},LabelOut=false,L=\hbox{$4$},x=3cm,y=0cm]{v3}
\Vertex[style={minimum
size=1.0cm,draw=cv4,fill=cfv4,text=clv4,shape=circle},LabelOut=false,L=\hbox{$9$},x=-.5cm,y=1.5cm]{v4}
\Vertex[style={minimum
size=1.0cm,draw=cv5,fill=cfv5,text=clv5,shape=circle},LabelOut=false,L=\hbox{$6$},x=1cm,y=0.0cm]{v5}
\Vertex[style={minimum
size=1.0cm,draw=cv6,fill=cfv6,text=clv6,shape=circle},LabelOut=false,L=\hbox{$5$},x=1cm,y=3cm]{v6}
\Vertex[style={minimum
size=1.0cm,draw=cv7,fill=cfv7,text=clv7,shape=circle},LabelOut=false,L=\hbox{$7$},x=7cm,y=3cm]{v7}
\Vertex[style={minimum
size=1.0cm,draw=cv8,fill=cfv8,text=clv8,shape=circle},LabelOut=false,L=\hbox{$8$},x=7cm,y=0cm]{v8}
\Edge[lw=0.1cm,style={color=cv0v3,},](v0)(v3)
\Edge[lw=0.1cm,style={color=cv0v4,},](v0)(v4)
\Edge[lw=0.1cm,style={color=cv0v7,},](v0)(v7)
\Edge[lw=0.1cm,style={color=cv0v8,},](v0)(v8)
\Edge[lw=0.1cm,style={color=cv1v2,},](v1)(v2)
\Edge[lw=0.1cm,style={color=cv1v4,},](v1)(v4)
\Edge[lw=0.1cm,style={color=cv1v7,},](v1)(v7)
\Edge[lw=0.1cm,style={color=cv1v8,},](v1)(v8)
\Edge[lw=0.1cm,style={color=cv2v4,},](v2)(v4)
\Edge[lw=0.1cm,style={color=cv2v5,},](v2)(v5)
\Edge[lw=0.1cm,style={color=cv2v6,},](v2)(v6)
\Edge[lw=0.1cm,style={color=cv3v4,},](v3)(v4)
\Edge[lw=0.1cm,style={color=cv3v5,},](v3)(v5)
\Edge[lw=0.1cm,style={color=cv3v6,},](v3)(v6)
\Edge[lw=0.1cm,style={color=cv4v5,},](v4)(v5)
\Edge[lw=0.1cm,style={color=cv4v6,},](v4)(v6)
\Edge[lw=0.1cm,style={color=cv7v8,},](v7)(v8)
\Edge[lw=0.1cm,style={dashed},](v2)(v0)
\Edge[lw=0.1cm,style={dashed},](v3)(v1)
\Edge[lw=0.1cm,style={color=red!70},](v2)(v3)
\end{tikzpicture}
\begin{tikzpicture}[scale=0.7]
\definecolor{cv0}{rgb}{0.0,0.0,0.0}
\definecolor{cfv0}{rgb}{1.0,1.0,1.0}
\definecolor{clv0}{rgb}{0.0,0.0,0.0}
\definecolor{cv1}{rgb}{0.0,0.0,0.0}
\definecolor{cfv1}{rgb}{1.0,1.0,1.0}
\definecolor{clv1}{rgb}{0.0,0.0,0.0}
\definecolor{cv2}{rgb}{0.0,0.0,0.0}
\definecolor{cfv2}{rgb}{1.0,1.0,1.0}
\definecolor{clv2}{rgb}{0.0,0.0,0.0}
\definecolor{cv3}{rgb}{0.0,0.0,0.0}
\definecolor{cfv3}{rgb}{1.0,1.0,1.0}
\definecolor{clv3}{rgb}{0.0,0.0,0.0}
\definecolor{cv4}{rgb}{0.0,0.0,0.0}
\definecolor{cfv4}{rgb}{1.0,1.0,1.0}
\definecolor{clv4}{rgb}{0.0,0.0,0.0}
\definecolor{cv5}{rgb}{0.0,0.0,0.0}
\definecolor{cfv5}{rgb}{1.0,1.0,1.0}
\definecolor{clv5}{rgb}{0.0,0.0,0.0}
\definecolor{cv6}{rgb}{0.0,0.0,0.0}
\definecolor{cfv6}{rgb}{1.0,1.0,1.0}
\definecolor{clv6}{rgb}{0.0,0.0,0.0}
\definecolor{cv7}{rgb}{0.0,0.0,0.0}
\definecolor{cfv7}{rgb}{1.0,1.0,1.0}
\definecolor{clv7}{rgb}{0.0,0.0,0.0}
\definecolor{cv8}{rgb}{0.0,0.0,0.0}
\definecolor{cfv8}{rgb}{1.0,1.0,1.0}
\definecolor{clv8}{rgb}{0.0,0.0,0.0}
\definecolor{cv0v3}{rgb}{0.0,0.0,0.0}
\definecolor{cv0v4}{rgb}{0.0,0.0,0.0}
\definecolor{cv0v7}{rgb}{0.0,0.0,0.0}
\definecolor{cv0v8}{rgb}{0.0,0.0,0.0}
\definecolor{cv1v2}{rgb}{0.0,0.0,0.0}
\definecolor{cv1v4}{rgb}{0.0,0.0,0.0}
\definecolor{cv1v7}{rgb}{0.0,0.0,0.0}
\definecolor{cv1v8}{rgb}{0.0,0.0,0.0}
\definecolor{cv2v4}{rgb}{0.0,0.0,0.0}
\definecolor{cv2v5}{rgb}{0.0,0.0,0.0}
\definecolor{cv2v6}{rgb}{0.0,0.0,0.0}
\definecolor{cv3v4}{rgb}{0.0,0.0,0.0}
\definecolor{cv3v5}{rgb}{0.0,0.0,0.0}
\definecolor{cv3v6}{rgb}{0.0,0.0,0.0}
\definecolor{cv4v5}{rgb}{0.0,0.0,0.0}
\definecolor{cv4v6}{rgb}{0.0,0.0,0.0}
\definecolor{cv7v8}{rgb}{0.0,0.0,0.0}
\Vertex[style={minimum
size=1.0cm,draw=cv0,fill=cfv0,text=clv0,shape=circle},LabelOut=false,L=\hbox{$2$},x=5cm,y=0cm]{v0}
\Vertex[style={minimum
size=1.0cm,draw=cv1,fill=cfv1,text=clv1,shape=circle},LabelOut=false,L=\hbox{$1$},x=5cm,y=3cm]{v1}
\Vertex[style={minimum
size=1.0cm,draw=cv2,fill=cfv2,text=clv2,shape=circle},LabelOut=false,L=\hbox{$3$},x=3cm,y=3cm]{v2}
\Vertex[style={minimum
size=1.0cm,draw=cv3,fill=cfv3,text=clv3,shape=circle},LabelOut=false,L=\hbox{$4$},x=3cm,y=0cm]{v3}
\Vertex[style={minimum
size=1.0cm,draw=cv4,fill=cfv4,text=clv4,shape=circle},LabelOut=false,L=\hbox{$9$},x=-.5cm,y=1.5cm]{v4}
\Vertex[style={minimum
size=1.0cm,draw=cv5,fill=cfv5,text=clv5,shape=circle},LabelOut=false,L=\hbox{$6$},x=1cm,y=0.0cm]{v5}
\Vertex[style={minimum
size=1.0cm,draw=cv6,fill=cfv6,text=clv6,shape=circle},LabelOut=false,L=\hbox{$5$},x=1cm,y=3cm]{v6}
\Vertex[style={minimum
size=1.0cm,draw=cv7,fill=cfv7,text=clv7,shape=circle},LabelOut=false,L=\hbox{$7$},x=7cm,y=3cm]{v7}
\Vertex[style={minimum
size=1.0cm,draw=cv8,fill=cfv8,text=clv8,shape=circle},LabelOut=false,L=\hbox{$8$},x=7cm,y=0cm]{v8}
\Edge[lw=0.1cm,style={color=cv0v3,},](v0)(v3)
\Edge[lw=0.1cm,style={color=cv0v4,},](v0)(v4)
\Edge[lw=0.1cm,style={color=cv0v7,},](v0)(v7)
\Edge[lw=0.1cm,style={color=cv0v8,},](v0)(v8)
\Edge[lw=0.1cm,style={color=cv1v2,},](v1)(v2)
\Edge[lw=0.1cm,style={color=cv1v4,},](v1)(v4)
\Edge[lw=0.1cm,style={color=cv1v7,},](v1)(v7)
\Edge[lw=0.1cm,style={color=cv1v8,},](v1)(v8)
\Edge[lw=0.1cm,style={color=cv2v4,},](v2)(v4)
\Edge[lw=0.1cm,style={color=cv2v5,},](v2)(v5)
\Edge[lw=0.1cm,style={color=cv2v6,},](v2)(v6)
\Edge[lw=0.1cm,style={color=cv3v4,},](v3)(v4)
\Edge[lw=0.1cm,style={color=cv3v5,},](v3)(v5)
\Edge[lw=0.1cm,style={color=cv3v6,},](v3)(v6)
\Edge[lw=0.1cm,style={color=cv4v5,},](v4)(v5)
\Edge[lw=0.1cm,style={color=cv4v6,},](v4)(v6)
\Edge[lw=0.1cm,style={color=cv7v8,},](v7)(v8)
\Edge[lw=0.1cm,style={dashed},](v2)(v0)
\Edge[lw=0.1cm,style={dashed},](v3)(v1)
\Edge[lw=0.1cm,style={color=red!70},](v1)(v0)
\end{tikzpicture}
    \caption{Three $\NDL$-cospectral pairs on 9 vertices}
    \label{fig:3pairscospec9}
\end{figure}

Again, we see these cospectral pairs may be constructed using a $\DL$-cospectrality construction from \cite{BDHLRSY19}. Let $G$ be a graph of order at least five with $v_1,v_2,v_3,v_4\in V(G)$. Let $C=\{\{v_1,v_2\},\{v_3,v_4\}\}$ and $U(C)=V(G)\backslash \{v_1,v_2,v_3,v_4\}$. Then $C$ is a set of {\em cousins} in $G$ if for all $u\in U(G)$, $d_G(u,v_1)=d_G(u,v_2)$, $d_G(u,v_3)=d_G(u,v_4)$, and $\sum_{u\in U(C)} d_G(u,v_1)=\sum_{u\in U(C)} d_G(u,v_3)$. 

\begin{thm}\textnormal{\cite[Theorem 3.10]{BDHLRSY19}}\label{cousins}
Let $G$ be a graph with a set of cousins $C=\{\{v_1,v_2\},\{v_3,v_4\}\}$ satisfying the following conditions:
\begin{itemize}
    \item $v_1v_2,v_3v_4\not\in E(G)$
    \item the subgraph of $G+v_1v_2$ induced by $\{v_1,v_2,v_3,v_4\}$ is isomorphic to the subgraph of $G+v_3v_4$ induced by $\{v_1,v_2,v_3,v_4\}$.
\end{itemize}
If $G+v_1v_2$ and $G+v_3v_4$ are not isomorphic, then they are $\DL$-cospectral.
\end{thm}

One can easily verify that in each of the three base graphs in Figure \ref{fig:3pairscospec9}, $C=\{\{1,2\},\{3,4\}\}$ is a set of cousins and the subgraph of $G+\{1,2\}$ induced by $\{1,2,3,4\}$ is isomorphic to the subgraph of $G+\{3,4\}$ induced by $\{1,2,3,4\}$. Therefore all three pairs of cospectral graphs can be constructed by Theorem \ref{cousins} and so are $\DL$-cospectral. When $0$ dashed edges are included, their $\DL$ characteristic polynomial is $p_{\DL}(x)=x^9 - 116x^8 + 5853x^7 - 167806x^6 + 2990335x^5 - 33920980x^4 +
239222875x^3 - 959072786x^2 + 1673692704x$. When 1 dashed edge is included, their $\DL$ characteristic polynomial is $p_{\DL}(x)=x^9 - 114x^8 + 5648x^7 - 158862x^6 + 2774997x^5 - 30830726x^4 +
212786586x^3 - 834230170x^2 + 1422606240x$. When both dashed edges are included, their $\DL$ characteristic polynomial is $p_{\DL}(x)=x^9 - 112x^8 + 5447x^7 - 150274x^6 + 2572751x^5 - 27995116x^4 + 189113161x^3 - 725242914x^2 + 1209121056x
$.

Again, we can see this does not always work. The pair of graphs in Figure \ref{fig:cousins} can be constructed in the way described in Theorem \ref{cousins} using their base graph and the set of cousins $\{\{1,2\},\{3,4\}\}$, so they are $\DL$-cospectral. However, they are not $\NDL$-cospectral.


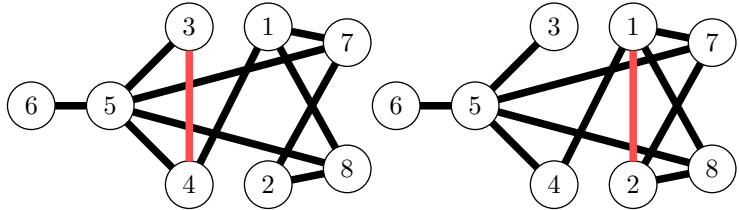
\begin{figure}[ht]
    \centering
\begin{tikzpicture}[scale=0.7]
\definecolor{cv0}{rgb}{0.0,0.0,0.0}
\definecolor{cfv0}{rgb}{1.0,1.0,1.0}
\definecolor{clv0}{rgb}{0.0,0.0,0.0}
\definecolor{cv1}{rgb}{0.0,0.0,0.0}
\definecolor{cfv1}{rgb}{1.0,1.0,1.0}
\definecolor{clv1}{rgb}{0.0,0.0,0.0}
\definecolor{cv2}{rgb}{0.0,0.0,0.0}
\definecolor{cfv2}{rgb}{1.0,1.0,1.0}
\definecolor{clv2}{rgb}{0.0,0.0,0.0}
\definecolor{cv3}{rgb}{0.0,0.0,0.0}
\definecolor{cfv3}{rgb}{1.0,1.0,1.0}
\definecolor{clv3}{rgb}{0.0,0.0,0.0}
\definecolor{cv4}{rgb}{0.0,0.0,0.0}
\definecolor{cfv4}{rgb}{1.0,1.0,1.0}
\definecolor{clv4}{rgb}{0.0,0.0,0.0}
\definecolor{cv5}{rgb}{0.0,0.0,0.0}
\definecolor{cfv5}{rgb}{1.0,1.0,1.0}
\definecolor{clv5}{rgb}{0.0,0.0,0.0}
\definecolor{cv6}{rgb}{0.0,0.0,0.0}
\definecolor{cfv6}{rgb}{1.0,1.0,1.0}
\definecolor{clv6}{rgb}{0.0,0.0,0.0}
\definecolor{cv7}{rgb}{0.0,0.0,0.0}
\definecolor{cfv7}{rgb}{1.0,1.0,1.0}
\definecolor{clv7}{rgb}{0.0,0.0,0.0}
\definecolor{cv0v3}{rgb}{0.0,0.0,0.0}
\definecolor{cv0v4}{rgb}{0.0,0.0,0.0}
\definecolor{cv0v5}{rgb}{0.0,0.0,0.0}
\definecolor{cv0v6}{rgb}{0.0,0.0,0.0}
\definecolor{cv0v7}{rgb}{0.0,0.0,0.0}
\definecolor{cv1v3}{rgb}{0.0,0.0,0.0}
\definecolor{cv1v4}{rgb}{0.0,0.0,0.0}
\definecolor{cv1v5}{rgb}{0.0,0.0,0.0}
\definecolor{cv2v3}{rgb}{0.0,0.0,0.0}
\definecolor{cv2v4}{rgb}{0.0,0.0,0.0}
\definecolor{cv5v6}{rgb}{0.0,0.0,0.0}
\Vertex[style={minimum
size=1.0cm,draw=cv0,fill=cfv0,text=clv0,shape=circle},LabelOut=false,L=\hbox{$5$},x=1.5cm,y=1.5cm]{v0}
\Vertex[style={minimum
size=1.0cm,draw=cv1,fill=cfv1,text=clv1,shape=circle},LabelOut=false,L=\hbox{$1$},x=4.5cm,y=3cm]{v1}
\Vertex[style={minimum
size=1.0cm,draw=cv2,fill=cfv2,text=clv2,shape=circle},LabelOut=false,L=\hbox{$2$},x=4.5cm,y=0cm]{v2}
\Vertex[style={minimum
size=1.0cm,draw=cv3,fill=cfv3,text=clv3,shape=circle},LabelOut=false,L=\hbox{$7$},x=6cm,y=2.7cm]{v3}
\Vertex[style={minimum
size=1.0cm,draw=cv4,fill=cfv4,text=clv4,shape=circle},LabelOut=false,L=\hbox{$8$},x=6cm,y=.3cm]{v4}
\Vertex[style={minimum
size=1.0cm,draw=cv5,fill=cfv5,text=clv5,shape=circle},LabelOut=false,L=\hbox{$4$},x=3cm,y=0cm]{v5}
\Vertex[style={minimum
size=1.0cm,draw=cv6,fill=cfv6,text=clv6,shape=circle},LabelOut=false,L=\hbox{$3$},x=3cm,y=3cm]{v6}
\Vertex[style={minimum
size=1.0cm,draw=cv7,fill=cfv7,text=clv7,shape=circle},LabelOut=false,L=\hbox{$6$},x=0,y=1.5cm]{v7}
\Edge[lw=0.1cm,style={color=cv0v3,},](v0)(v3)
\Edge[lw=0.1cm,style={color=cv0v4,},](v0)(v4)
\Edge[lw=0.1cm,style={color=cv0v5,},](v0)(v5)
\Edge[lw=0.1cm,style={color=cv0v6,},](v0)(v6)
\Edge[lw=0.1cm,style={color=cv0v7,},](v0)(v7)
\Edge[lw=0.1cm,style={color=cv1v3,},](v1)(v3)
\Edge[lw=0.1cm,style={color=cv1v4,},](v1)(v4)
\Edge[lw=0.1cm,style={color=cv1v5,},](v1)(v5)
\Edge[lw=0.1cm,style={color=cv2v3,},](v2)(v3)
\Edge[lw=0.1cm,style={color=cv2v4,},](v2)(v4)
\Edge[lw=0.1cm,style={color=red!70,},](v5)(v6)
\end{tikzpicture}
\begin{tikzpicture}[scale=0.7]
\definecolor{cv0}{rgb}{0.0,0.0,0.0}
\definecolor{cfv0}{rgb}{1.0,1.0,1.0}
\definecolor{clv0}{rgb}{0.0,0.0,0.0}
\definecolor{cv1}{rgb}{0.0,0.0,0.0}
\definecolor{cfv1}{rgb}{1.0,1.0,1.0}
\definecolor{clv1}{rgb}{0.0,0.0,0.0}
\definecolor{cv2}{rgb}{0.0,0.0,0.0}
\definecolor{cfv2}{rgb}{1.0,1.0,1.0}
\definecolor{clv2}{rgb}{0.0,0.0,0.0}
\definecolor{cv3}{rgb}{0.0,0.0,0.0}
\definecolor{cfv3}{rgb}{1.0,1.0,1.0}
\definecolor{clv3}{rgb}{0.0,0.0,0.0}
\definecolor{cv4}{rgb}{0.0,0.0,0.0}
\definecolor{cfv4}{rgb}{1.0,1.0,1.0}
\definecolor{clv4}{rgb}{0.0,0.0,0.0}
\definecolor{cv5}{rgb}{0.0,0.0,0.0}
\definecolor{cfv5}{rgb}{1.0,1.0,1.0}
\definecolor{clv5}{rgb}{0.0,0.0,0.0}
\definecolor{cv6}{rgb}{0.0,0.0,0.0}
\definecolor{cfv6}{rgb}{1.0,1.0,1.0}
\definecolor{clv6}{rgb}{0.0,0.0,0.0}
\definecolor{cv7}{rgb}{0.0,0.0,0.0}
\definecolor{cfv7}{rgb}{1.0,1.0,1.0}
\definecolor{clv7}{rgb}{0.0,0.0,0.0}
\definecolor{cv0v3}{rgb}{0.0,0.0,0.0}
\definecolor{cv0v4}{rgb}{0.0,0.0,0.0}
\definecolor{cv0v5}{rgb}{0.0,0.0,0.0}
\definecolor{cv0v6}{rgb}{0.0,0.0,0.0}
\definecolor{cv0v7}{rgb}{0.0,0.0,0.0}
\definecolor{cv1v3}{rgb}{0.0,0.0,0.0}
\definecolor{cv1v4}{rgb}{0.0,0.0,0.0}
\definecolor{cv1v5}{rgb}{0.0,0.0,0.0}
\definecolor{cv2v3}{rgb}{0.0,0.0,0.0}
\definecolor{cv2v4}{rgb}{0.0,0.0,0.0}
\definecolor{cv5v6}{rgb}{0.0,0.0,0.0}
\Vertex[style={minimum
size=1.0cm,draw=cv0,fill=cfv0,text=clv0,shape=circle},LabelOut=false,L=\hbox{$5$},x=1.5cm,y=1.5cm]{v0}
\Vertex[style={minimum
size=1.0cm,draw=cv1,fill=cfv1,text=clv1,shape=circle},LabelOut=false,L=\hbox{$1$},x=4.5cm,y=3cm]{v1}
\Vertex[style={minimum
size=1.0cm,draw=cv2,fill=cfv2,text=clv2,shape=circle},LabelOut=false,L=\hbox{$2$},x=4.5cm,y=0cm]{v2}
\Vertex[style={minimum
size=1.0cm,draw=cv3,fill=cfv3,text=clv3,shape=circle},LabelOut=false,L=\hbox{$7$},x=6cm,y=2.7cm]{v3}
\Vertex[style={minimum
size=1.0cm,draw=cv4,fill=cfv4,text=clv4,shape=circle},LabelOut=false,L=\hbox{$8$},x=6cm,y=.3cm]{v4}
\Vertex[style={minimum
size=1.0cm,draw=cv5,fill=cfv5,text=clv5,shape=circle},LabelOut=false,L=\hbox{$4$},x=3cm,y=0cm]{v5}
\Vertex[style={minimum
size=1.0cm,draw=cv6,fill=cfv6,text=clv6,shape=circle},LabelOut=false,L=\hbox{$3$},x=3cm,y=3cm]{v6}
\Vertex[style={minimum
size=1.0cm,draw=cv7,fill=cfv7,text=clv7,shape=circle},LabelOut=false,L=\hbox{$6$},x=0,y=1.5cm]{v7}
\Edge[lw=0.1cm,style={color=cv0v3,},](v0)(v3)
\Edge[lw=0.1cm,style={color=cv0v4,},](v0)(v4)
\Edge[lw=0.1cm,style={color=cv0v5,},](v0)(v5)
\Edge[lw=0.1cm,style={color=cv0v6,},](v0)(v6)
\Edge[lw=0.1cm,style={color=cv0v7,},](v0)(v7)
\Edge[lw=0.1cm,style={color=cv1v3,},](v1)(v3)
\Edge[lw=0.1cm,style={color=cv1v4,},](v1)(v4)
\Edge[lw=0.1cm,style={color=cv1v5,},](v1)(v5)
\Edge[lw=0.1cm,style={color=cv2v3,},](v2)(v3)
\Edge[lw=0.1cm,style={color=cv2v4,},](v2)(v4)
\Edge[lw=0.1cm,style={color=red!70,},](v1)(v2)
\end{tikzpicture}
    \caption{Graphs that are $\DL$ but not $\NDL$-cospectral using the cousins construction}
    \label{fig:cousins}
\end{figure}

The last $\NDL$-cospectral pair on 9 vertices is shown in Figure \ref{fig:differentedges}. While the other three $\NDL$-cospectral pairs involve edge switching, in this pair two additional edges ({\red light} colored) are added to the first graph to obtain the second. The $\NDL$ characteristic polynomial of the graphs is $p_{\NDL}(x)=x^9 - 9x^8 + \frac{23884}{675}x^7 -\frac{7232482}{91125}x^6 + \frac{30369859}{273375}x^5 -\frac{27161183}{273375}x^4 + \frac{15156922}{273375}x^3 -\frac{1608332}{91125}x^2 +\frac{74536}{30375}x $. The graphs in Figure \ref{fig:differentedges} show that the number of edges, the degree sequence, and the transmission sequence are not preserved by $\NDL$-cospectrality. The {\em degree sequence} of a graph $G$ is the list of degrees of the vertices in $G$ and the {\em transmission sequence} of a graph $G$ is the list of transmissions of the vertices in $G$. The transmission sequence of $G_1$ is $[9, 9, 9, 9, 10, 10, 10, 10, 12]$ and the transmission sequence of $G_2$ is $[9, 9, 9, 9, 9, 9, 10, 10, 10]$. The degree sequence of $G_1$ is $[4, 6, 6, 6, 6,7, 7, 7, 7]$ and the degree sequence of $G_2$ is $[6,6,6,7, 7, 7, 7, 7, 7]$. This pair also provides an example of a cospectral pair that is $\NDL$-cospectral but not $\DL$-cospectral. 


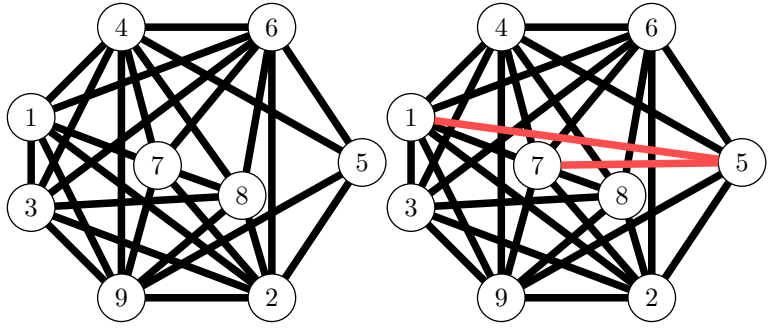
\begin{figure}[ht]
    \centering
\begin{tikzpicture}[scale=0.8]
\definecolor{cv0}{rgb}{0.0,0.0,0.0}
\definecolor{cfv0}{rgb}{1.0,1.0,1.0}
\definecolor{clv0}{rgb}{0.0,0.0,0.0}
\definecolor{cv1}{rgb}{0.0,0.0,0.0}
\definecolor{cfv1}{rgb}{1.0,1.0,1.0}
\definecolor{clv1}{rgb}{0.0,0.0,0.0}
\definecolor{cv2}{rgb}{0.0,0.0,0.0}
\definecolor{cfv2}{rgb}{1.0,1.0,1.0}
\definecolor{clv2}{rgb}{0.0,0.0,0.0}
\definecolor{cv3}{rgb}{0.0,0.0,0.0}
\definecolor{cfv3}{rgb}{1.0,1.0,1.0}
\definecolor{clv3}{rgb}{0.0,0.0,0.0}
\definecolor{cv4}{rgb}{0.0,0.0,0.0}
\definecolor{cfv4}{rgb}{1.0,1.0,1.0}
\definecolor{clv4}{rgb}{0.0,0.0,0.0}
\definecolor{cv5}{rgb}{0.0,0.0,0.0}
\definecolor{cfv5}{rgb}{1.0,1.0,1.0}
\definecolor{clv5}{rgb}{0.0,0.0,0.0}
\definecolor{cv6}{rgb}{0.0,0.0,0.0}
\definecolor{cfv6}{rgb}{1.0,1.0,1.0}
\definecolor{clv6}{rgb}{0.0,0.0,0.0}
\definecolor{cv7}{rgb}{0.0,0.0,0.0}
\definecolor{cfv7}{rgb}{1.0,1.0,1.0}
\definecolor{clv7}{rgb}{0.0,0.0,0.0}
\definecolor{cv8}{rgb}{0.0,0.0,0.0}
\definecolor{cfv8}{rgb}{1.0,1.0,1.0}
\definecolor{clv8}{rgb}{0.0,0.0,0.0}
\definecolor{cv0v1}{rgb}{0.0,0.0,0.0}
\definecolor{cv0v2}{rgb}{0.0,0.0,0.0}
\definecolor{cv0v3}{rgb}{0.0,0.0,0.0}
\definecolor{cv0v4}{rgb}{0.0,0.0,0.0}
\definecolor{cv0v5}{rgb}{0.0,0.0,0.0}
\definecolor{cv0v7}{rgb}{0.0,0.0,0.0}
\definecolor{cv0v8}{rgb}{0.0,0.0,0.0}
\definecolor{cv1v2}{rgb}{0.0,0.0,0.0}
\definecolor{cv1v3}{rgb}{0.0,0.0,0.0}
\definecolor{cv1v4}{rgb}{0.0,0.0,0.0}
\definecolor{cv1v6}{rgb}{0.0,0.0,0.0}
\definecolor{cv1v7}{rgb}{0.0,0.0,0.0}
\definecolor{cv2v3}{rgb}{0.0,0.0,0.0}
\definecolor{cv2v5}{rgb}{0.0,0.0,0.0}
\definecolor{cv2v6}{rgb}{0.0,0.0,0.0}
\definecolor{cv2v7}{rgb}{0.0,0.0,0.0}
\definecolor{cv2v8}{rgb}{0.0,0.0,0.0}
\definecolor{cv3v4}{rgb}{0.0,0.0,0.0}
\definecolor{cv3v6}{rgb}{0.0,0.0,0.0}
\definecolor{cv3v8}{rgb}{0.0,0.0,0.0}
\definecolor{cv4v5}{rgb}{0.0,0.0,0.0}
\definecolor{cv4v6}{rgb}{0.0,0.0,0.0}
\definecolor{cv4v7}{rgb}{0.0,0.0,0.0}
\definecolor{cv4v8}{rgb}{0.0,0.0,0.0}
\definecolor{cv5v6}{rgb}{0.0,0.0,0.0}
\definecolor{cv6v7}{rgb}{0.0,0.0,0.0}
\definecolor{cv6v8}{rgb}{0.0,0.0,0.0}
\definecolor{cv7v8}{rgb}{0.0,0.0,0.0}
\Vertex[style={minimum
size=1.0cm,draw=cv0,fill=cfv0,text=clv0,shape=circle},LabelOut=false,L=\hbox{$9$},x=1cm,y=0.0cm]{v0}
\Vertex[style={minimum
size=1.0cm,draw=cv1,fill=cfv1,text=clv1,shape=circle},LabelOut=false,L=\hbox{$1$},x=-.5cm,y=3cm]{v1}
\Vertex[style={minimum
size=1.0cm,draw=cv2,fill=cfv2,text=clv2,shape=circle},LabelOut=false,L=\hbox{$2$},x=3.5cm,y=0cm]{v2}
\Vertex[style={minimum
size=1.0cm,draw=cv3,fill=cfv3,text=clv3,shape=circle},LabelOut=false,L=\hbox{$3$},x=-.5cm,y=1.5cm]{v3}
\Vertex[style={minimum
size=1.0cm,draw=cv4,fill=cfv4,text=clv4,shape=circle},LabelOut=false,L=\hbox{$4$},x=1cm,y=4.5cm]{v4}
\Vertex[style={minimum
size=1.0cm,draw=cv5,fill=cfv5,text=clv5,shape=circle},LabelOut=false,L=\hbox{$5$},x=5cm,y=2.25cm]{v5}
\Vertex[style={minimum
size=1.0cm,draw=cv6,fill=cfv6,text=clv6,shape=circle},LabelOut=false,L=\hbox{$6$},x=3.5cm,y=4.5cm]{v6}
\Vertex[style={minimum
size=1.0cm,draw=cv7,fill=cfv7,text=clv7,shape=circle},LabelOut=false,L=\hbox{$7$},x=1.6cm,y=2.2cm]{v7}
\Vertex[style={minimum
size=1.0cm,draw=cv8,fill=cfv8,text=clv8,shape=circle},LabelOut=false,L=\hbox{$8$},x=3.0057cm,y=1.7cm]{v8}
\Edge[lw=0.1cm,style={color=cv0v1,},](v0)(v1)
\Edge[lw=0.1cm,style={color=cv0v2,},](v0)(v2)
\Edge[lw=0.1cm,style={color=cv0v3,},](v0)(v3)
\Edge[lw=0.1cm,style={color=cv0v4,},](v0)(v4)
\Edge[lw=0.1cm,style={color=cv0v5,},](v0)(v5)
\Edge[lw=0.1cm,style={color=cv0v7,},](v0)(v7)
\Edge[lw=0.1cm,style={color=cv0v8,},](v0)(v8)
\Edge[lw=0.1cm,style={color=cv1v2,},](v1)(v2)
\Edge[lw=0.1cm,style={color=cv1v3,},](v1)(v3)
\Edge[lw=0.1cm,style={color=cv1v4,},](v1)(v4)
\Edge[lw=0.1cm,style={color=cv1v6,},](v1)(v6)
\Edge[lw=0.1cm,style={color=cv1v7,},](v1)(v7)
\Edge[lw=0.1cm,style={color=cv2v3,},](v2)(v3)
\Edge[lw=0.1cm,style={color=cv2v5,},](v2)(v5)
\Edge[lw=0.1cm,style={color=cv2v6,},](v2)(v6)
\Edge[lw=0.1cm,style={color=cv2v7,},](v2)(v7)
\Edge[lw=0.1cm,style={color=cv2v8,},](v2)(v8)
\Edge[lw=0.1cm,style={color=cv3v4,},](v3)(v4)
\Edge[lw=0.1cm,style={color=cv3v6,},](v3)(v6)
\Edge[lw=0.1cm,style={color=cv3v8,},](v3)(v8)
\Edge[lw=0.1cm,style={color=cv4v5,},](v4)(v5)
\Edge[lw=0.1cm,style={color=cv4v6,},](v4)(v6)
\Edge[lw=0.1cm,style={color=cv4v7,},](v4)(v7)
\Edge[lw=0.1cm,style={color=cv4v8,},](v4)(v8)
\Edge[lw=0.1cm,style={color=cv5v6,},](v5)(v6)
\Edge[lw=0.1cm,style={color=cv6v7,},](v6)(v7)
\Edge[lw=0.1cm,style={color=cv6v8,},](v6)(v8)
\Edge[lw=0.1cm,style={color=cv7v8,},](v7)(v8)
\end{tikzpicture}
\begin{tikzpicture}[scale=0.8]
\definecolor{cv0}{rgb}{0.0,0.0,0.0}
\definecolor{cfv0}{rgb}{1.0,1.0,1.0}
\definecolor{clv0}{rgb}{0.0,0.0,0.0}
\definecolor{cv1}{rgb}{0.0,0.0,0.0}
\definecolor{cfv1}{rgb}{1.0,1.0,1.0}
\definecolor{clv1}{rgb}{0.0,0.0,0.0}
\definecolor{cv2}{rgb}{0.0,0.0,0.0}
\definecolor{cfv2}{rgb}{1.0,1.0,1.0}
\definecolor{clv2}{rgb}{0.0,0.0,0.0}
\definecolor{cv3}{rgb}{0.0,0.0,0.0}
\definecolor{cfv3}{rgb}{1.0,1.0,1.0}
\definecolor{clv3}{rgb}{0.0,0.0,0.0}
\definecolor{cv4}{rgb}{0.0,0.0,0.0}
\definecolor{cfv4}{rgb}{1.0,1.0,1.0}
\definecolor{clv4}{rgb}{0.0,0.0,0.0}
\definecolor{cv5}{rgb}{0.0,0.0,0.0}
\definecolor{cfv5}{rgb}{1.0,1.0,1.0}
\definecolor{clv5}{rgb}{0.0,0.0,0.0}
\definecolor{cv6}{rgb}{0.0,0.0,0.0}
\definecolor{cfv6}{rgb}{1.0,1.0,1.0}
\definecolor{clv6}{rgb}{0.0,0.0,0.0}
\definecolor{cv7}{rgb}{0.0,0.0,0.0}
\definecolor{cfv7}{rgb}{1.0,1.0,1.0}
\definecolor{clv7}{rgb}{0.0,0.0,0.0}
\definecolor{cv8}{rgb}{0.0,0.0,0.0}
\definecolor{cfv8}{rgb}{1.0,1.0,1.0}
\definecolor{clv8}{rgb}{0.0,0.0,0.0}
\definecolor{cv0v1}{rgb}{0.0,0.0,0.0}
\definecolor{cv0v2}{rgb}{0.0,0.0,0.0}
\definecolor{cv0v3}{rgb}{0.0,0.0,0.0}
\definecolor{cv0v4}{rgb}{0.0,0.0,0.0}
\definecolor{cv0v5}{rgb}{0.0,0.0,0.0}
\definecolor{cv0v7}{rgb}{0.0,0.0,0.0}
\definecolor{cv0v8}{rgb}{0.0,0.0,0.0}
\definecolor{cv1v2}{rgb}{0.0,0.0,0.0}
\definecolor{cv1v3}{rgb}{0.0,0.0,0.0}
\definecolor{cv1v4}{rgb}{0.0,0.0,0.0}
\definecolor{cv1v6}{rgb}{0.0,0.0,0.0}
\definecolor{cv1v7}{rgb}{0.0,0.0,0.0}
\definecolor{cv2v3}{rgb}{0.0,0.0,0.0}
\definecolor{cv2v5}{rgb}{0.0,0.0,0.0}
\definecolor{cv2v6}{rgb}{0.0,0.0,0.0}
\definecolor{cv2v7}{rgb}{0.0,0.0,0.0}
\definecolor{cv2v8}{rgb}{0.0,0.0,0.0}
\definecolor{cv3v4}{rgb}{0.0,0.0,0.0}
\definecolor{cv3v6}{rgb}{0.0,0.0,0.0}
\definecolor{cv3v8}{rgb}{0.0,0.0,0.0}
\definecolor{cv4v5}{rgb}{0.0,0.0,0.0}
\definecolor{cv4v6}{rgb}{0.0,0.0,0.0}
\definecolor{cv4v7}{rgb}{0.0,0.0,0.0}
\definecolor{cv4v8}{rgb}{0.0,0.0,0.0}
\definecolor{cv5v6}{rgb}{0.0,0.0,0.0}
\definecolor{cv6v7}{rgb}{0.0,0.0,0.0}
\definecolor{cv6v8}{rgb}{0.0,0.0,0.0}
\definecolor{cv7v8}{rgb}{0.0,0.0,0.0}
\Vertex[style={minimum
size=1.0cm,draw=cv0,fill=cfv0,text=clv0,shape=circle},LabelOut=false,L=\hbox{$9$},x=1cm,y=0.0cm]{v0}
\Vertex[style={minimum
size=1.0cm,draw=cv1,fill=cfv1,text=clv1,shape=circle},LabelOut=false,L=\hbox{$1$},x=-.5cm,y=3cm]{v1}
\Vertex[style={minimum
size=1.0cm,draw=cv2,fill=cfv2,text=clv2,shape=circle},LabelOut=false,L=\hbox{$2$},x=3.5cm,y=0cm]{v2}
\Vertex[style={minimum
size=1.0cm,draw=cv3,fill=cfv3,text=clv3,shape=circle},LabelOut=false,L=\hbox{$3$},x=-.5cm,y=1.5cm]{v3}
\Vertex[style={minimum
size=1.0cm,draw=cv4,fill=cfv4,text=clv4,shape=circle},LabelOut=false,L=\hbox{$4$},x=1cm,y=4.5cm]{v4}
\Vertex[style={minimum
size=1.0cm,draw=cv5,fill=cfv5,text=clv5,shape=circle},LabelOut=false,L=\hbox{$5$},x=5cm,y=2.25cm]{v5}
\Vertex[style={minimum
size=1.0cm,draw=cv6,fill=cfv6,text=clv6,shape=circle},LabelOut=false,L=\hbox{$6$},x=3.5cm,y=4.5cm]{v6}
\Vertex[style={minimum
size=1.0cm,draw=cv7,fill=cfv7,text=clv7,shape=circle},LabelOut=false,L=\hbox{$7$},x=1.6cm,y=2.2cm]{v7}
\Vertex[style={minimum
size=1.0cm,draw=cv8,fill=cfv8,text=clv8,shape=circle},LabelOut=false,L=\hbox{$8$},x=3.0057cm,y=1.7cm]{v8}
\Edge[lw=0.1cm,style={color=cv0v1,},](v0)(v1)
\Edge[lw=0.1cm,style={color=cv0v2,},](v0)(v2)
\Edge[lw=0.1cm,style={color=cv0v3,},](v0)(v3)
\Edge[lw=0.1cm,style={color=cv0v4,},](v0)(v4)
\Edge[lw=0.1cm,style={color=cv0v5,},](v0)(v5)
\Edge[lw=0.1cm,style={color=cv0v7,},](v0)(v7)
\Edge[lw=0.1cm,style={color=cv0v8,},](v0)(v8)
\Edge[lw=0.1cm,style={color=cv1v2,},](v1)(v2)
\Edge[lw=0.1cm,style={color=cv1v3,},](v1)(v3)
\Edge[lw=0.1cm,style={color=cv1v4,},](v1)(v4)
\Edge[lw=0.1cm,style={color=cv1v6,},](v1)(v6)
\Edge[lw=0.1cm,style={color=cv1v7,},](v1)(v7)
\Edge[lw=0.1cm,style={color=cv2v3,},](v2)(v3)
\Edge[lw=0.1cm,style={color=cv2v5,},](v2)(v5)
\Edge[lw=0.1cm,style={color=cv2v6,},](v2)(v6)
\Edge[lw=0.1cm,style={color=cv2v7,},](v2)(v7)
\Edge[lw=0.1cm,style={color=cv2v8,},](v2)(v8)
\Edge[lw=0.1cm,style={color=cv3v4,},](v3)(v4)
\Edge[lw=0.1cm,style={color=cv3v6,},](v3)(v6)
\Edge[lw=0.1cm,style={color=cv3v8,},](v3)(v8)
\Edge[lw=0.1cm,style={color=cv4v5,},](v4)(v5)
\Edge[lw=0.1cm,style={color=cv4v6,},](v4)(v6)
\Edge[lw=0.1cm,style={color=cv4v7,},](v4)(v7)
\Edge[lw=0.1cm,style={color=cv4v8,},](v4)(v8)
\Edge[lw=0.1cm,style={color=cv5v6,},](v5)(v6)
\Edge[lw=0.1cm,style={color=cv6v7,},](v6)(v7)
\Edge[lw=0.1cm,style={color=cv6v8,},](v6)(v8)
\Edge[lw=0.1cm,style={color=cv7v8,},](v7)(v8)
\Edge[lw=0.1cm,style={color=red!70,},](v1)(v5)
\Edge[lw=0.1cm,style={color=red!70,},](v7)(v5)
\end{tikzpicture}
    \caption{$G_1$ and $G_2$, $\NDL$-cospectral pair on 9 vertices with a different number of edges, different degree sequences, and different transmission sequences}
    \label{fig:differentedges}
\end{figure}

On 10 vertices, there are 3763 pairs of $\NDL$-cospectral graphs and 4 triples. A graph is {\em planar} if it can be drawn in the plane without any edges crossing. In Figure \ref{fig:cospecPlanar}, $H_1$ is planar and $H_2$ is not, and the graphs share the $\NDL$ characteristic polynomial $p_{\NDL}(x)= x^{10} - 10x^9 + \frac{11760575}{264992}x^8 - \frac{14698252437}{128107070}x^7 +\frac{1561159495967}{8198852480}x^6 - \frac{5605798973451}{26646270560}x^5 +\frac{7068694654043}{45679320960}x^4 - \frac{11682868723247}{159877623360}x^3 +
\frac{535639931153}{26646270560}x^2 - \frac{65401424433}{26646270560}x$.
Therefore planarity is not preserved by $\NDL$-cospectrality.

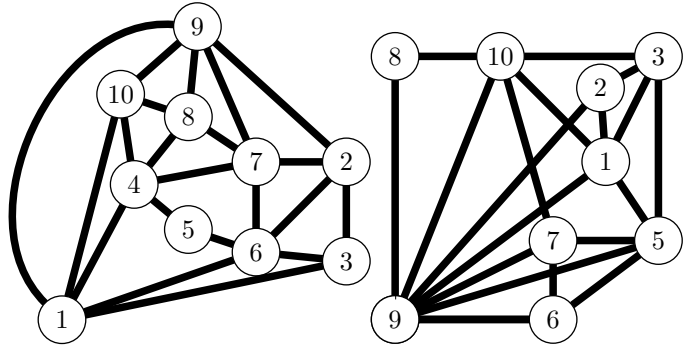
\begin{figure}[h!]
\begin{center}
\begin{tikzpicture}[scale=0.6]
\definecolor{cv0}{rgb}{0.0,0.0,0.0}
\definecolor{cfv0}{rgb}{1.0,1.0,1.0}
\definecolor{clv0}{rgb}{0.0,0.0,0.0}
\definecolor{cv1}{rgb}{0.0,0.0,0.0}
\definecolor{cfv1}{rgb}{1.0,1.0,1.0}
\definecolor{clv1}{rgb}{0.0,0.0,0.0}
\definecolor{cv2}{rgb}{0.0,0.0,0.0}
\definecolor{cfv2}{rgb}{1.0,1.0,1.0}
\definecolor{clv2}{rgb}{0.0,0.0,0.0}
\definecolor{cv3}{rgb}{0.0,0.0,0.0}
\definecolor{cfv3}{rgb}{1.0,1.0,1.0}
\definecolor{clv3}{rgb}{0.0,0.0,0.0}
\definecolor{cv4}{rgb}{0.0,0.0,0.0}
\definecolor{cfv4}{rgb}{1.0,1.0,1.0}
\definecolor{clv4}{rgb}{0.0,0.0,0.0}
\definecolor{cv5}{rgb}{0.0,0.0,0.0}
\definecolor{cfv5}{rgb}{1.0,1.0,1.0}
\definecolor{clv5}{rgb}{0.0,0.0,0.0}
\definecolor{cv6}{rgb}{0.0,0.0,0.0}
\definecolor{cfv6}{rgb}{1.0,1.0,1.0}
\definecolor{clv6}{rgb}{0.0,0.0,0.0}
\definecolor{cv7}{rgb}{0.0,0.0,0.0}
\definecolor{cfv7}{rgb}{1.0,1.0,1.0}
\definecolor{clv7}{rgb}{0.0,0.0,0.0}
\definecolor{cv8}{rgb}{0.0,0.0,0.0}
\definecolor{cfv8}{rgb}{1.0,1.0,1.0}
\definecolor{clv8}{rgb}{0.0,0.0,0.0}
\definecolor{cv9}{rgb}{0.0,0.0,0.0}
\definecolor{cfv9}{rgb}{1.0,1.0,1.0}
\definecolor{clv9}{rgb}{0.0,0.0,0.0}
\definecolor{cv0v1}{rgb}{0.0,0.0,0.0}
\definecolor{cv0v4}{rgb}{0.0,0.0,0.0}
\definecolor{cv0v8}{rgb}{0.0,0.0,0.0}
\definecolor{cv0v9}{rgb}{0.0,0.0,0.0}
\definecolor{cv1v3}{rgb}{0.0,0.0,0.0}
\definecolor{cv1v4}{rgb}{0.0,0.0,0.0}
\definecolor{cv1v6}{rgb}{0.0,0.0,0.0}
\definecolor{cv1v9}{rgb}{0.0,0.0,0.0}
\definecolor{cv2v3}{rgb}{0.0,0.0,0.0}
\definecolor{cv2v6}{rgb}{0.0,0.0,0.0}
\definecolor{cv2v7}{rgb}{0.0,0.0,0.0}
\definecolor{cv2v9}{rgb}{0.0,0.0,0.0}
\definecolor{cv3v6}{rgb}{0.0,0.0,0.0}
\definecolor{cv4v5}{rgb}{0.0,0.0,0.0}
\definecolor{cv4v7}{rgb}{0.0,0.0,0.0}
\definecolor{cv4v8}{rgb}{0.0,0.0,0.0}
\definecolor{cv5v6}{rgb}{0.0,0.0,0.0}
\definecolor{cv6v7}{rgb}{0.0,0.0,0.0}
\definecolor{cv7v8}{rgb}{0.0,0.0,0.0}
\definecolor{cv7v9}{rgb}{0.0,0.0,0.0}
\definecolor{cv8v9}{rgb}{0.0,0.0,0.0}
\Vertex[style={minimum
size=1.0cm,draw=cv0,fill=cfv0,text=clv0,shape=circle},LabelOut=false,L=\hbox{$10$},x=0.0cm,y=4.5cm]{v0}
\Vertex[style={minimum
size=1.0cm,draw=cv1,fill=cfv1,text=clv1,shape=circle},LabelOut=false,L=\hbox{$1$},x=-1.3cm,y=-.5cm]{v1}
\Vertex[style={minimum
size=1.0cm,draw=cv2,fill=cfv2,text=clv2,shape=circle},LabelOut=false,L=\hbox{$2$},x=5cm,y=3cm]{v2}
\Vertex[style={minimum
size=1.0cm,draw=cv3,fill=cfv3,text=clv3,shape=circle},LabelOut=false,L=\hbox{$3$},x=5cm,y=.8cm]{v3}
\Vertex[style={minimum
size=1.0cm,draw=cv4,fill=cfv4,text=clv4,shape=circle},LabelOut=false,L=\hbox{$4$},x=.3cm,y=2.5cm]{v4}
\Vertex[style={minimum
size=1.0cm,draw=cv5,fill=cfv5,text=clv5,shape=circle},LabelOut=false,L=\hbox{$5$},x=1.5cm,y=1.5cm]{v5}
\Vertex[style={minimum
size=1.0cm,draw=cv6,fill=cfv6,text=clv6,shape=circle},LabelOut=false,L=\hbox{$6$},x=3cm,y=1cm]{v6}
\Vertex[style={minimum
size=1.0cm,draw=cv7,fill=cfv7,text=clv7,shape=circle},LabelOut=false,L=\hbox{$7$},x=3cm,y=3cm]{v7}
\Vertex[style={minimum
size=1.0cm,draw=cv8,fill=cfv8,text=clv8,shape=circle},LabelOut=false,L=\hbox{$8$},x=1.5cm,y=4cm]{v8}
\Vertex[style={minimum
size=1.0cm,draw=cv9,fill=cfv9,text=clv9,shape=circle},LabelOut=false,L=\hbox{$9$},x=1.7cm,y=6cm]{v9}
\Edge[lw=0.1cm,style={color=cv0v1,},](v0)(v1)
\Edge[lw=0.1cm,style={color=cv0v4,},](v0)(v4)
\Edge[lw=0.1cm,style={color=cv0v8,},](v0)(v8)
\Edge[lw=0.1cm,style={color=cv0v9,},](v0)(v9)
\Edge[lw=0.1cm,style={color=cv1v3,},](v1)(v3)
\Edge[lw=0.1cm,style={color=cv1v4,},](v1)(v4)
\Edge[lw=0.1cm,style={color=cv1v6,},](v1)(v6)
\Edge[lw=0.1cm,style={color=cv2v3,},](v2)(v3)
\Edge[lw=0.1cm,style={color=cv2v6,},](v2)(v6)
\Edge[lw=0.1cm,style={color=cv2v7,},](v2)(v7)
\Edge[lw=0.1cm,style={color=cv2v9,},](v2)(v9)
\Edge[lw=0.1cm,style={color=cv3v6,},](v3)(v6)
\Edge[lw=0.1cm,style={color=cv4v5,},](v4)(v5)
\Edge[lw=0.1cm,style={color=cv4v7,},](v4)(v7)
\Edge[lw=0.1cm,style={color=cv4v8,},](v4)(v8)
\Edge[lw=0.1cm,style={color=cv5v6,},](v5)(v6)
\Edge[lw=0.1cm,style={color=cv6v7,},](v6)(v7)
\Edge[lw=0.1cm,style={color=cv7v8,},](v7)(v8)
\Edge[lw=0.1cm,style={color=cv7v9,},](v7)(v9)
\Edge[lw=0.1cm,style={color=cv8v9,},](v8)(v9)
\tikzset{EdgeStyle/.append style = {bend left = 70}}
\Edge[lw=0.1cm,style={color=cv1v9,},](v1)(v9)
\end{tikzpicture}   	
\begin{tikzpicture}[scale=0.7]
\definecolor{cv0}{rgb}{0.0,0.0,0.0}
\definecolor{cfv0}{rgb}{1.0,1.0,1.0}
\definecolor{clv0}{rgb}{0.0,0.0,0.0}
\definecolor{cv1}{rgb}{0.0,0.0,0.0}
\definecolor{cfv1}{rgb}{1.0,1.0,1.0}
\definecolor{clv1}{rgb}{0.0,0.0,0.0}
\definecolor{cv2}{rgb}{0.0,0.0,0.0}
\definecolor{cfv2}{rgb}{1.0,1.0,1.0}
\definecolor{clv2}{rgb}{0.0,0.0,0.0}
\definecolor{cv3}{rgb}{0.0,0.0,0.0}
\definecolor{cfv3}{rgb}{1.0,1.0,1.0}
\definecolor{clv3}{rgb}{0.0,0.0,0.0}
\definecolor{cv4}{rgb}{0.0,0.0,0.0}
\definecolor{cfv4}{rgb}{1.0,1.0,1.0}
\definecolor{clv4}{rgb}{0.0,0.0,0.0}
\definecolor{cv5}{rgb}{0.0,0.0,0.0}
\definecolor{cfv5}{rgb}{1.0,1.0,1.0}
\definecolor{clv5}{rgb}{0.0,0.0,0.0}
\definecolor{cv6}{rgb}{0.0,0.0,0.0}
\definecolor{cfv6}{rgb}{1.0,1.0,1.0}
\definecolor{clv6}{rgb}{0.0,0.0,0.0}
\definecolor{cv7}{rgb}{0.0,0.0,0.0}
\definecolor{cfv7}{rgb}{1.0,1.0,1.0}
\definecolor{clv7}{rgb}{0.0,0.0,0.0}
\definecolor{cv8}{rgb}{0.0,0.0,0.0}
\definecolor{cfv8}{rgb}{1.0,1.0,1.0}
\definecolor{clv8}{rgb}{0.0,0.0,0.0}
\definecolor{cv9}{rgb}{0.0,0.0,0.0}
\definecolor{cfv9}{rgb}{1.0,1.0,1.0}
\definecolor{clv9}{rgb}{0.0,0.0,0.0}
\definecolor{cv0v1}{rgb}{0.0,0.0,0.0}
\definecolor{cv0v3}{rgb}{0.0,0.0,0.0}
\definecolor{cv0v4}{rgb}{0.0,0.0,0.0}
\definecolor{cv0v7}{rgb}{0.0,0.0,0.0}
\definecolor{cv0v8}{rgb}{0.0,0.0,0.0}
\definecolor{cv1v2}{rgb}{0.0,0.0,0.0}
\definecolor{cv1v3}{rgb}{0.0,0.0,0.0}
\definecolor{cv1v5}{rgb}{0.0,0.0,0.0}
\definecolor{cv1v9}{rgb}{0.0,0.0,0.0}
\definecolor{cv2v3}{rgb}{0.0,0.0,0.0}
\definecolor{cv2v9}{rgb}{0.0,0.0,0.0}
\definecolor{cv3v5}{rgb}{0.0,0.0,0.0}
\definecolor{cv4v5}{rgb}{0.0,0.0,0.0}
\definecolor{cv4v6}{rgb}{0.0,0.0,0.0}
\definecolor{cv4v7}{rgb}{0.0,0.0,0.0}
\definecolor{cv4v9}{rgb}{0.0,0.0,0.0}
\definecolor{cv5v6}{rgb}{0.0,0.0,0.0}
\definecolor{cv5v7}{rgb}{0.0,0.0,0.0}
\definecolor{cv6v7}{rgb}{0.0,0.0,0.0}
\definecolor{cv6v9}{rgb}{0.0,0.0,0.0}
\definecolor{cv8v9}{rgb}{0.0,0.0,0.0}
\Vertex[style={minimum
size=1.0cm,draw=cv0,fill=cfv0,text=clv0,shape=circle},LabelOut=false,L=\hbox{$10$},x=2cm,y=5cm]{v0}
\Vertex[style={minimum
size=1.0cm,draw=cv1,fill=cfv1,text=clv1,shape=circle},LabelOut=false,L=\hbox{$1$},x=4cm,y=3cm]{v1}
\Vertex[style={minimum
size=1.0cm,draw=cv2,fill=cfv2,text=clv2,shape=circle},LabelOut=false,L=\hbox{$2$},x=3.9cm,y=4.4cm]{v2}
\Vertex[style={minimum
size=1.0cm,draw=cv3,fill=cfv3,text=clv3,shape=circle},LabelOut=false,L=\hbox{$3$},x=5.0cm,y=5cm]{v3}
\Vertex[style={minimum
size=1.0cm,draw=cv4,fill=cfv4,text=clv4,shape=circle},LabelOut=false,L=\hbox{$4$},x=0cm,y=0cm]{v4}
\Vertex[style={minimum
size=1.0cm,draw=cv5,fill=cfv5,text=clv5,shape=circle},LabelOut=false,L=\hbox{$5$},x=5cm,y=1.5cm]{v5}
\Vertex[style={minimum
size=1.0cm,draw=cv6,fill=cfv6,text=clv6,shape=circle},LabelOut=false,L=\hbox{$6$},x=3cm,y=0cm]{v6}
\Vertex[style={minimum
size=1.0cm,draw=cv7,fill=cfv7,text=clv7,shape=circle},LabelOut=false,L=\hbox{$7$},x=3cm,y=1.5cm]{v7}
\Vertex[style={minimum
size=1.0cm,draw=cv8,fill=cfv8,text=clv8,shape=circle},LabelOut=false,L=\hbox{$8$},x=0cm,y=5cm]{v8}
\Vertex[style={minimum
size=1.0cm,draw=cv9,fill=cfv9,text=clv9,shape=circle},LabelOut=false,L=\hbox{$9$},x=0cm,y=0cm]{v9}
\Edge[lw=0.1cm,style={color=cv0v1,},](v0)(v1)
\Edge[lw=0.1cm,style={color=cv0v3,},](v0)(v3)
\Edge[lw=0.1cm,style={color=cv0v4,},](v0)(v4)
\Edge[lw=0.1cm,style={color=cv0v7,},](v0)(v7)
\Edge[lw=0.1cm,style={color=cv0v8,},](v0)(v8)
\Edge[lw=0.1cm,style={color=cv1v2,},](v1)(v2)
\Edge[lw=0.1cm,style={color=cv1v3,},](v1)(v3)
\Edge[lw=0.1cm,style={color=cv1v5,},](v1)(v5)
\Edge[lw=0.1cm,style={color=cv1v9,},](v1)(v9)
\Edge[lw=0.1cm,style={color=cv2v3,},](v2)(v3)
\Edge[lw=0.1cm,style={color=cv2v9,},](v2)(v9)
\Edge[lw=0.1cm,style={color=cv3v5,},](v3)(v5)
\Edge[lw=0.1cm,style={color=cv4v5,},](v4)(v5)
\Edge[lw=0.1cm,style={color=cv4v6,},](v4)(v6)
\Edge[lw=0.1cm,style={color=cv4v7,},](v4)(v7)
\Edge[lw=0.1cm,style={color=cv4v9,},](v4)(v9)
\Edge[lw=0.1cm,style={color=cv5v6,},](v5)(v6)
\Edge[lw=0.1cm,style={color=cv5v7,},](v5)(v7)
\Edge[lw=0.1cm,style={color=cv6v7,},](v6)(v7)
\Edge[lw=0.1cm,style={color=cv6v9,},](v6)(v9)
\Edge[lw=0.1cm,style={color=cv8v9,},](v8)(v9)
\end{tikzpicture}

    \caption{$H_1$ and $H_2$, $\NDL$ cospectral graphs where one is planar and one is not}
    \label{fig:cospecPlanar}
\end{center}
\end{figure}


The {\em Weiner index} of a graph $G$ with vertices $V(G)=\{v_1,\dots,v_n\}$ is\\ $W(G)=\frac{1}{2}\sum_{i=1}^{n}\sum_{j=1}^{n}d(v_i,v_j)$ and the {\em girth} of a graph is the length of the shortest cycle in the graph. The pair of graphs in Figure \ref{fig:cospecGirth} show girth, Wiener index, $k$-regularity, and $k$-transmission regularity are not preserved by $\NDL$-cospectrality. This is in contrast to $\DL$-cospectrality, which preserves the Weiner index (observe $W(G)=\frac{1}{2}\trace{\DL(G)}$ \cite{AH18}). While $F_1$ has girth 5, Weiner index 75, is $3$-regular and is $10$-transmission regular, $F_2$ has girth 3, Weiner index 50, is $8$-regular, and is $15$-transmission regular. Note that while this shows $k$-regularity and $k$-transmission regularity are not preserved by $\NDL$-cospectrality, both graphs are still regular and transmission regular so it does not show regularity or transmission regularity are not preserved. These graphs share the $\NDL$ characteristic polynomial $p_{\NDL}(x)=x^{10} - 10x^9 + \frac{222}{5}x^8 - \frac{2872}{25}x^7 + \frac{23861}{125}x^6 -\frac{660126}{3125}x^5 + \frac{486504}{3125}x^4 - \frac{230256}{3125}x^3 + \frac{63504}{3125}x^2 -
\frac{7776}{3125}x$. 

\begin{figure}[h!]
\begin{center}
       	
\begin{tikzpicture}[scale=0.75]
\definecolor{cv0}{rgb}{0.0,0.0,0.0}
\definecolor{cfv0}{rgb}{1.0,1.0,1.0}
\definecolor{clv0}{rgb}{0.0,0.0,0.0}
\definecolor{cv1}{rgb}{0.0,0.0,0.0}
\definecolor{cfv1}{rgb}{1.0,1.0,1.0}
\definecolor{clv1}{rgb}{0.0,0.0,0.0}
\definecolor{cv2}{rgb}{0.0,0.0,0.0}
\definecolor{cfv2}{rgb}{1.0,1.0,1.0}
\definecolor{clv2}{rgb}{0.0,0.0,0.0}
\definecolor{cv3}{rgb}{0.0,0.0,0.0}
\definecolor{cfv3}{rgb}{1.0,1.0,1.0}
\definecolor{clv3}{rgb}{0.0,0.0,0.0}
\definecolor{cv4}{rgb}{0.0,0.0,0.0}
\definecolor{cfv4}{rgb}{1.0,1.0,1.0}
\definecolor{clv4}{rgb}{0.0,0.0,0.0}
\definecolor{cv5}{rgb}{0.0,0.0,0.0}
\definecolor{cfv5}{rgb}{1.0,1.0,1.0}
\definecolor{clv5}{rgb}{0.0,0.0,0.0}
\definecolor{cv6}{rgb}{0.0,0.0,0.0}
\definecolor{cfv6}{rgb}{1.0,1.0,1.0}
\definecolor{clv6}{rgb}{0.0,0.0,0.0}
\definecolor{cv7}{rgb}{0.0,0.0,0.0}
\definecolor{cfv7}{rgb}{1.0,1.0,1.0}
\definecolor{clv7}{rgb}{0.0,0.0,0.0}
\definecolor{cv8}{rgb}{0.0,0.0,0.0}
\definecolor{cfv8}{rgb}{1.0,1.0,1.0}
\definecolor{clv8}{rgb}{0.0,0.0,0.0}
\definecolor{cv9}{rgb}{0.0,0.0,0.0}
\definecolor{cfv9}{rgb}{1.0,1.0,1.0}
\definecolor{clv9}{rgb}{0.0,0.0,0.0}
\definecolor{cv0v1}{rgb}{0.0,0.0,0.0}
\definecolor{cv0v8}{rgb}{0.0,0.0,0.0}
\definecolor{cv0v9}{rgb}{0.0,0.0,0.0}
\definecolor{cv1v3}{rgb}{0.0,0.0,0.0}
\definecolor{cv1v6}{rgb}{0.0,0.0,0.0}
\definecolor{cv2v3}{rgb}{0.0,0.0,0.0}
\definecolor{cv2v7}{rgb}{0.0,0.0,0.0}
\definecolor{cv2v9}{rgb}{0.0,0.0,0.0}
\definecolor{cv3v4}{rgb}{0.0,0.0,0.0}
\definecolor{cv4v5}{rgb}{0.0,0.0,0.0}
\definecolor{cv4v8}{rgb}{0.0,0.0,0.0}
\definecolor{cv5v6}{rgb}{0.0,0.0,0.0}
\definecolor{cv5v9}{rgb}{0.0,0.0,0.0}
\definecolor{cv6v7}{rgb}{0.0,0.0,0.0}
\definecolor{cv7v8}{rgb}{0.0,0.0,0.0}
\Vertex[style={minimum
size=1.0cm,draw=cv0,fill=cfv0,text=clv0,shape=circle},LabelOut=false,L=\hbox{$10$},x=5cm,y=1cm]{v0}
\Vertex[style={minimum
size=1.0cm,draw=cv1,fill=cfv1,text=clv1,shape=circle},LabelOut=false,L=\hbox{$1$},x=5.0cm,y=5cm]{v1}
\Vertex[style={minimum
size=1.0cm,draw=cv2,fill=cfv2,text=clv2,shape=circle},LabelOut=false,L=\hbox{$2$},x=1cm,y=5cm]{v2}
\Vertex[style={minimum
size=1.0cm,draw=cv3,fill=cfv3,text=clv3,shape=circle},LabelOut=false,L=\hbox{$3$},x=3cm,y=6cm]{v3}
\Vertex[style={minimum
size=1.0cm,draw=cv4,fill=cfv4,text=clv4,shape=circle},LabelOut=false,L=\hbox{$4$},x=3cm,y=2cm]{v4}
\Vertex[style={minimum
size=1.0cm,draw=cv5,fill=cfv5,text=clv5,shape=circle},LabelOut=false,L=\hbox{$5$},x=2.5cm,y=3cm]{v5}
\Vertex[style={minimum
size=1.0cm,draw=cv6,fill=cfv6,text=clv6,shape=circle},LabelOut=false,L=\hbox{$6$},x=0cm,y=3cm]{v6}x=1cm,y=3.2cm
\Vertex[style={minimum
size=1.0cm,draw=cv7,fill=cfv7,text=clv7,shape=circle},LabelOut=false,L=\hbox{$7$},x=1cm,y=1cm]{v7}
\Vertex[style={minimum
size=1.0cm,draw=cv8,fill=cfv8,text=clv8,shape=circle},LabelOut=false,L=\hbox{$8$},x=3cm,y=0.0cm]{v8}
\Vertex[style={minimum
size=1.0cm,draw=cv9,fill=cfv9,text=clv9,shape=circle},LabelOut=false,L=\hbox{$9$},x=6cm,y=3cm]{v9}
\Edge[lw=0.1cm,style={color=cv0v1,},](v0)(v1)
\Edge[lw=0.1cm,style={color=cv0v8,},](v0)(v8)
\Edge[lw=0.1cm,style={color=cv0v9,},](v0)(v9)
\Edge[lw=0.1cm,style={color=cv1v3,},](v1)(v3)
\Edge[lw=0.1cm,style={color=cv1v6,},](v1)(v6)
\Edge[lw=0.1cm,style={color=cv2v3,},](v2)(v3)
\Edge[lw=0.1cm,style={color=cv2v7,},](v2)(v7)
\Edge[lw=0.1cm,style={color=cv2v9,},](v2)(v9)
\Edge[lw=0.1cm,style={color=cv3v4,},](v3)(v4)
\Edge[lw=0.1cm,style={color=cv4v5,},](v4)(v5)
\Edge[lw=0.1cm,style={color=cv4v8,},](v4)(v8)
\Edge[lw=0.1cm,style={color=cv5v6,},](v5)(v6)
\Edge[lw=0.1cm,style={color=cv5v9,},](v5)(v9)
\Edge[lw=0.1cm,style={color=cv6v7,},](v6)(v7)
\Edge[lw=0.1cm,style={color=cv7v8,},](v7)(v8)
\end{tikzpicture}
\begin{tikzpicture}[scale=0.8]
\definecolor{cv0}{rgb}{0.0,0.0,0.0}
\definecolor{cfv0}{rgb}{1.0,1.0,1.0}
\definecolor{clv0}{rgb}{0.0,0.0,0.0}
\definecolor{cv1}{rgb}{0.0,0.0,0.0}
\definecolor{cfv1}{rgb}{1.0,1.0,1.0}
\definecolor{clv1}{rgb}{0.0,0.0,0.0}
\definecolor{cv2}{rgb}{0.0,0.0,0.0}
\definecolor{cfv2}{rgb}{1.0,1.0,1.0}
\definecolor{clv2}{rgb}{0.0,0.0,0.0}
\definecolor{cv3}{rgb}{0.0,0.0,0.0}
\definecolor{cfv3}{rgb}{1.0,1.0,1.0}
\definecolor{clv3}{rgb}{0.0,0.0,0.0}
\definecolor{cv4}{rgb}{0.0,0.0,0.0}
\definecolor{cfv4}{rgb}{1.0,1.0,1.0}
\definecolor{clv4}{rgb}{0.0,0.0,0.0}
\definecolor{cv5}{rgb}{0.0,0.0,0.0}
\definecolor{cfv5}{rgb}{1.0,1.0,1.0}
\definecolor{clv5}{rgb}{0.0,0.0,0.0}
\definecolor{cv6}{rgb}{0.0,0.0,0.0}
\definecolor{cfv6}{rgb}{1.0,1.0,1.0}
\definecolor{clv6}{rgb}{0.0,0.0,0.0}
\definecolor{cv7}{rgb}{0.0,0.0,0.0}
\definecolor{cfv7}{rgb}{1.0,1.0,1.0}
\definecolor{clv7}{rgb}{0.0,0.0,0.0}
\definecolor{cv8}{rgb}{0.0,0.0,0.0}
\definecolor{cfv8}{rgb}{1.0,1.0,1.0}
\definecolor{clv8}{rgb}{0.0,0.0,0.0}
\definecolor{cv9}{rgb}{0.0,0.0,0.0}
\definecolor{cfv9}{rgb}{1.0,1.0,1.0}
\definecolor{clv9}{rgb}{0.0,0.0,0.0}
\definecolor{cv0v1}{rgb}{0.0,0.0,0.0}
\definecolor{cv0v2}{rgb}{0.0,0.0,0.0}
\definecolor{cv0v3}{rgb}{0.0,0.0,0.0}
\definecolor{cv0v4}{rgb}{0.0,0.0,0.0}
\definecolor{cv0v5}{rgb}{0.0,0.0,0.0}
\definecolor{cv0v7}{rgb}{0.0,0.0,0.0}
\definecolor{cv0v8}{rgb}{0.0,0.0,0.0}
\definecolor{cv0v9}{rgb}{0.0,0.0,0.0}
\definecolor{cv1v2}{rgb}{0.0,0.0,0.0}
\definecolor{cv1v3}{rgb}{0.0,0.0,0.0}
\definecolor{cv1v4}{rgb}{0.0,0.0,0.0}
\definecolor{cv1v5}{rgb}{0.0,0.0,0.0}
\definecolor{cv1v6}{rgb}{0.0,0.0,0.0}
\definecolor{cv1v8}{rgb}{0.0,0.0,0.0}
\definecolor{cv1v9}{rgb}{0.0,0.0,0.0}
\definecolor{cv2v3}{rgb}{0.0,0.0,0.0}
\definecolor{cv2v5}{rgb}{0.0,0.0,0.0}
\definecolor{cv2v6}{rgb}{0.0,0.0,0.0}
\definecolor{cv2v7}{rgb}{0.0,0.0,0.0}
\definecolor{cv2v8}{rgb}{0.0,0.0,0.0}
\definecolor{cv2v9}{rgb}{0.0,0.0,0.0}
\definecolor{cv3v4}{rgb}{0.0,0.0,0.0}
\definecolor{cv3v5}{rgb}{0.0,0.0,0.0}
\definecolor{cv3v6}{rgb}{0.0,0.0,0.0}
\definecolor{cv3v7}{rgb}{0.0,0.0,0.0}
\definecolor{cv3v9}{rgb}{0.0,0.0,0.0}
\definecolor{cv4v5}{rgb}{0.0,0.0,0.0}
\definecolor{cv4v6}{rgb}{0.0,0.0,0.0}
\definecolor{cv4v7}{rgb}{0.0,0.0,0.0}
\definecolor{cv4v8}{rgb}{0.0,0.0,0.0}
\definecolor{cv4v9}{rgb}{0.0,0.0,0.0}
\definecolor{cv5v6}{rgb}{0.0,0.0,0.0}
\definecolor{cv5v7}{rgb}{0.0,0.0,0.0}
\definecolor{cv5v8}{rgb}{0.0,0.0,0.0}
\definecolor{cv6v7}{rgb}{0.0,0.0,0.0}
\definecolor{cv6v8}{rgb}{0.0,0.0,0.0}
\definecolor{cv6v9}{rgb}{0.0,0.0,0.0}
\definecolor{cv7v8}{rgb}{0.0,0.0,0.0}
\definecolor{cv7v9}{rgb}{0.0,0.0,0.0}
\definecolor{cv8v9}{rgb}{0.0,0.0,0.0}
\Vertex[style={minimum size=1.0cm,draw=cv0,fill=cfv0,text=clv0,shape=circle},LabelOut=false,L=\hbox{$10$},x=5cm,y=5.0cm]{v0}
\Vertex[style={minimum size=1.0cm,draw=cv1,fill=cfv1,text=clv1,shape=circle},LabelOut=false,L=\hbox{$1$},x=3cm,y=6cm]{v1}
\Vertex[style={minimum size=1.0cm,draw=cv2,fill=cfv2,text=clv2,shape=circle},LabelOut=false,L=\hbox{$2$},x=1cm,y=5cm]{v2}
\Vertex[style={minimum size=1.0cm,draw=cv3,fill=cfv3,text=clv3,shape=circle},LabelOut=false,L=\hbox{$3$},x=0cm,y=3cm]{v3}
\Vertex[style={minimum size=1.0cm,draw=cv4,fill=cfv4,text=clv4,shape=circle},LabelOut=false,L=\hbox{$4$},x=5cm,y=1cm]{v4}
\Vertex[style={minimum size=1.0cm,draw=cv5,fill=cfv5,text=clv5,shape=circle},LabelOut=false,L=\hbox{$5$},x=6cm,y=3cm]{v5}
\Vertex[style={minimum size=1.0cm,draw=cv6,fill=cfv6,text=clv6,shape=circle},LabelOut=false,L=\hbox{$6$},x=1cm,y=1cm]{v6}
\Vertex[style={minimum size=1.0cm,draw=cv7,fill=cfv7,text=clv7,shape=circle},LabelOut=false,L=\hbox{$7$},x=3cm,y=0cm]{v7}
\Vertex[style={minimum size=1.0cm,draw=cv8,fill=cfv8,text=clv8,shape=circle},LabelOut=false,L=\hbox{$8$},x=3.5cm,y=3.5cm]{v8}
\Vertex[style={minimum size=1.0cm,draw=cv9,fill=cfv9,text=clv9,shape=circle},LabelOut=false,L=\hbox{$9$},x=2.5cm,y=3.5cm]{v9}
\Edge[lw=0.1cm,style={color=cv0v1,},](v0)(v1)
\Edge[lw=0.1cm,style={color=cv0v2,},](v0)(v2)
\Edge[lw=0.1cm,style={color=cv0v3,},](v0)(v3)
\Edge[lw=0.1cm,style={color=cv0v4,},](v0)(v4)
\Edge[lw=0.1cm,style={color=cv0v5,},](v0)(v5)
\Edge[lw=0.1cm,style={color=cv0v7,},](v0)(v7)
\Edge[lw=0.1cm,style={color=cv0v8,},](v0)(v8)
\Edge[lw=0.1cm,style={color=cv0v9,},](v0)(v9)
\Edge[lw=0.1cm,style={color=cv1v2,},](v1)(v2)
\Edge[lw=0.1cm,style={color=cv1v3,},](v1)(v3)
\Edge[lw=0.1cm,style={color=cv1v4,},](v1)(v4)
\Edge[lw=0.1cm,style={color=cv1v5,},](v1)(v5)
\Edge[lw=0.1cm,style={color=cv1v6,},](v1)(v6)
\Edge[lw=0.1cm,style={color=cv1v8,},](v1)(v8)
\Edge[lw=0.1cm,style={color=cv1v9,},](v1)(v9)
\Edge[lw=0.1cm,style={color=cv2v3,},](v2)(v3)
\Edge[lw=0.1cm,style={color=cv2v5,},](v2)(v5)
\Edge[lw=0.1cm,style={color=cv2v6,},](v2)(v6)
\Edge[lw=0.1cm,style={color=cv2v7,},](v2)(v7)
\Edge[lw=0.1cm,style={color=cv2v8,},](v2)(v8)
\Edge[lw=0.1cm,style={color=cv2v9,},](v2)(v9)
\Edge[lw=0.1cm,style={color=cv3v4,},](v3)(v4)
\Edge[lw=0.1cm,style={color=cv3v5,},](v3)(v5)
\Edge[lw=0.1cm,style={color=cv3v6,},](v3)(v6)
\Edge[lw=0.1cm,style={color=cv3v7,},](v3)(v7)
\Edge[lw=0.1cm,style={color=cv3v9,},](v3)(v9)
\Edge[lw=0.1cm,style={color=cv4v5,},](v4)(v5)
\Edge[lw=0.1cm,style={color=cv4v6,},](v4)(v6)
\Edge[lw=0.1cm,style={color=cv4v7,},](v4)(v7)
\Edge[lw=0.1cm,style={color=cv4v8,},](v4)(v8)
\Edge[lw=0.1cm,style={color=cv4v9,},](v4)(v9)
\Edge[lw=0.1cm,style={color=cv5v6,},](v5)(v6)
\Edge[lw=0.1cm,style={color=cv5v7,},](v5)(v7)
\Edge[lw=0.1cm,style={color=cv5v8,},](v5)(v8)
\Edge[lw=0.1cm,style={color=cv6v7,},](v6)(v7)
\Edge[lw=0.1cm,style={color=cv6v8,},](v6)(v8)
\Edge[lw=0.1cm,style={color=cv6v9,},](v6)(v9)
\Edge[lw=0.1cm,style={color=cv7v8,},](v7)(v8)
\Edge[lw=0.1cm,style={color=cv7v9,},](v7)(v9)
\Edge[lw=0.1cm,style={color=cv8v9,},](v8)(v9)
\end{tikzpicture}
    \caption{$F_1$ and $F_2$, $\NDL$ cospectral graphs with different girth, Weiner indexes, $k$-regularity, and $k$-transmission regularity and are only $M$-cospectral for $M=\NDL$}
    \label{fig:cospecGirth}
\end{center}
\end{figure}


The {\em diameter} of a graph $G$ is the maximum distance between any pair of vertices in the graph. In \cite{EGM04}, the authors show that for $r$-regular graphs with diameter at most 2, if $\lambda_n\leq \dots\leq \lambda_2\leq \lambda_1=r$, then $\partial_1=2n-2-r$ and $\partial_i=-\lambda_i-2$ for $2\leq i\leq n$. So in the case of graphs with diameter at most two that are regular and transmission regular, the eigenvalues of $A,L,Q,\cL,\D,\DL,$ and $\DQ$ can all be obtained from each other using the above result and Observations \ref{Obs:Reg} and \ref{Obs:TransReg}. Since both of the graphs in Figure \ref{fig:cospecGirth} are regular, transmission regular, and have diameter $2$, we can easily calculate their spectra for $A,L,Q,\cL,\D,\DL,$ and $\DQ$ from their $\NDL$-spectrum. However, since the two graphs have different regularity and different transmission regularity, it is clear that their spectra will be different for every other matrix. Therefore $F_1$ and $F_2$ serve as an example of graphs that are only cospectral with respect to $\NDL$.

We can also find a pair of graphs that are $M$-cospectral for all $M=A,L,Q,\cL,\D,\DL,\DQ$ and $\NDL$. In Figure \ref{fig:cospecALL}, the two graphs are both diameter 2, 5-regular, and 13-transmission regular and they are $\NDL$-cospectral, so they will be $M$-cospectral for all $M=A,L,Q,\cL,\D,\DL$, and $\DQ$ as well. We can also see this by noting that $\phi(\lambda,r,L_1)=\phi(\lambda,r,L_2)$ and $\phi^{\D}(\lambda,r,L_1)=\phi^{\D}(\lambda,r,L_2)$.

\begin{figure}[h!]
\begin{center}
       	
\begin{tikzpicture}[scale=0.75]
\definecolor{cv0}{rgb}{0.0,0.0,0.0}
\definecolor{cfv0}{rgb}{1.0,1.0,1.0}
\definecolor{clv0}{rgb}{0.0,0.0,0.0}
\definecolor{cv1}{rgb}{0.0,0.0,0.0}
\definecolor{cfv1}{rgb}{1.0,1.0,1.0}
\definecolor{clv1}{rgb}{0.0,0.0,0.0}
\definecolor{cv2}{rgb}{0.0,0.0,0.0}
\definecolor{cfv2}{rgb}{1.0,1.0,1.0}
\definecolor{clv2}{rgb}{0.0,0.0,0.0}
\definecolor{cv3}{rgb}{0.0,0.0,0.0}
\definecolor{cfv3}{rgb}{1.0,1.0,1.0}
\definecolor{clv3}{rgb}{0.0,0.0,0.0}
\definecolor{cv4}{rgb}{0.0,0.0,0.0}
\definecolor{cfv4}{rgb}{1.0,1.0,1.0}
\definecolor{clv4}{rgb}{0.0,0.0,0.0}
\definecolor{cv5}{rgb}{0.0,0.0,0.0}
\definecolor{cfv5}{rgb}{1.0,1.0,1.0}
\definecolor{clv5}{rgb}{0.0,0.0,0.0}
\definecolor{cv6}{rgb}{0.0,0.0,0.0}
\definecolor{cfv6}{rgb}{1.0,1.0,1.0}
\definecolor{clv6}{rgb}{0.0,0.0,0.0}
\definecolor{cv7}{rgb}{0.0,0.0,0.0}
\definecolor{cfv7}{rgb}{1.0,1.0,1.0}
\definecolor{clv7}{rgb}{0.0,0.0,0.0}
\definecolor{cv8}{rgb}{0.0,0.0,0.0}
\definecolor{cfv8}{rgb}{1.0,1.0,1.0}
\definecolor{clv8}{rgb}{0.0,0.0,0.0}
\definecolor{cv9}{rgb}{0.0,0.0,0.0}
\definecolor{cfv9}{rgb}{1.0,1.0,1.0}
\definecolor{clv9}{rgb}{0.0,0.0,0.0}
\definecolor{cv0v1}{rgb}{0.0,0.0,0.0}
\definecolor{cv0v2}{rgb}{0.0,0.0,0.0}
\definecolor{cv0v5}{rgb}{0.0,0.0,0.0}
\definecolor{cv0v8}{rgb}{0.0,0.0,0.0}
\definecolor{cv0v9}{rgb}{0.0,0.0,0.0}
\definecolor{cv1v2}{rgb}{0.0,0.0,0.0}
\definecolor{cv1v3}{rgb}{0.0,0.0,0.0}
\definecolor{cv1v5}{rgb}{0.0,0.0,0.0}
\definecolor{cv1v8}{rgb}{0.0,0.0,0.0}
\definecolor{cv2v3}{rgb}{0.0,0.0,0.0}
\definecolor{cv2v7}{rgb}{0.0,0.0,0.0}
\definecolor{cv2v9}{rgb}{0.0,0.0,0.0}
\definecolor{cv3v4}{rgb}{0.0,0.0,0.0}
\definecolor{cv3v6}{rgb}{0.0,0.0,0.0}
\definecolor{cv3v9}{rgb}{0.0,0.0,0.0}
\definecolor{cv4v5}{rgb}{0.0,0.0,0.0}
\definecolor{cv4v7}{rgb}{0.0,0.0,0.0}
\definecolor{cv4v8}{rgb}{0.0,0.0,0.0}
\definecolor{cv4v9}{rgb}{0.0,0.0,0.0}
\definecolor{cv5v6}{rgb}{0.0,0.0,0.0}
\definecolor{cv5v7}{rgb}{0.0,0.0,0.0}
\definecolor{cv6v7}{rgb}{0.0,0.0,0.0}
\definecolor{cv6v8}{rgb}{0.0,0.0,0.0}
\definecolor{cv6v9}{rgb}{0.0,0.0,0.0}
\definecolor{cv7v8}{rgb}{0.0,0.0,0.0}
\Vertex[style={minimum
size=1.0cm,draw=cv0,fill=cfv0,text=clv0,shape=circle},LabelOut=false,L=\hbox{$10$},x=3.5cm,y=5cm]{v0}
\Vertex[style={minimum
size=1.0cm,draw=cv1,fill=cfv1,text=clv1,shape=circle},LabelOut=false,L=\hbox{$1$},x=5cm,y=4cm]{v1}
\Vertex[style={minimum
size=1.0cm,draw=cv2,fill=cfv2,text=clv2,shape=circle},LabelOut=false,L=\hbox{$2$},x=1.5cm,y=5cm]{v2}
\Vertex[style={minimum
size=1.0cm,draw=cv3,fill=cfv3,text=clv3,shape=circle},LabelOut=false,L=\hbox{$3$},x=1cm,y=3.8cm]{v3}
\Vertex[style={minimum
size=1.0cm,draw=cv4,fill=cfv4,text=clv4,shape=circle},LabelOut=false,L=\hbox{$4$},x=1cm,y=2cm]{v4}
\Vertex[style={minimum
size=1.0cm,draw=cv5,fill=cfv5,text=clv5,shape=circle},LabelOut=false,L=\hbox{$5$},x=5cm,y=0cm]{v5}
\Vertex[style={minimum
size=1.0cm,draw=cv6,fill=cfv6,text=clv6,shape=circle},LabelOut=false,L=\hbox{$6$},x=0cm,y=0cm]{v6}
\Vertex[style={minimum
size=1.0cm,draw=cv7,fill=cfv7,text=clv7,shape=circle},LabelOut=false,L=\hbox{$7$},x=3.5cm,y=1.5cm]{v7}
\Vertex[style={minimum
size=1.0cm,draw=cv8,fill=cfv8,text=clv8,shape=circle},LabelOut=false,L=\hbox{$8$},x=3.5cm,y=3cm]{v8}
\Vertex[style={minimum
size=1.0cm,draw=cv9,fill=cfv9,text=clv9,shape=circle},LabelOut=false,L=\hbox{$9$},x=0cm,y=4cm]{v9}
\Edge[lw=0.1cm,style={color=cv0v1,},](v0)(v1)
\Edge[lw=0.1cm,style={color=cv0v2,},](v0)(v2)
\Edge[lw=0.1cm,style={color=cv0v5,},](v0)(v5)
\Edge[lw=0.1cm,style={color=cv0v8,},](v0)(v8)
\Edge[lw=0.1cm,style={color=cv0v9,},](v0)(v9)
\Edge[lw=0.1cm,style={color=cv1v2,},](v1)(v2)
\Edge[lw=0.1cm,style={color=cv1v3,},](v1)(v3)
\Edge[lw=0.1cm,style={color=cv1v5,},](v1)(v5)
\Edge[lw=0.1cm,style={color=cv1v8,},](v1)(v8)
\Edge[lw=0.1cm,style={color=cv2v3,},](v2)(v3)
\Edge[lw=0.1cm,style={color=cv2v7,},](v2)(v7)
\Edge[lw=0.1cm,style={color=cv2v9,},](v2)(v9)
\Edge[lw=0.1cm,style={color=cv3v4,},](v3)(v4)
\Edge[lw=0.1cm,style={color=cv3v6,},](v3)(v6)
\Edge[lw=0.1cm,style={color=cv3v9,},](v3)(v9)
\Edge[lw=0.1cm,style={color=cv4v5,},](v4)(v5)
\Edge[lw=0.1cm,style={color=cv4v7,},](v4)(v7)
\Edge[lw=0.1cm,style={color=cv4v8,},](v4)(v8)
\Edge[lw=0.1cm,style={color=cv4v9,},](v4)(v9)
\Edge[lw=0.1cm,style={color=cv5v6,},](v5)(v6)
\Edge[lw=0.1cm,style={color=cv5v7,},](v5)(v7)
\Edge[lw=0.1cm,style={color=cv6v7,},](v6)(v7)
\Edge[lw=0.1cm,style={color=cv6v8,},](v6)(v8)
\Edge[lw=0.1cm,style={color=cv6v9,},](v6)(v9)
\Edge[lw=0.1cm,style={color=cv7v8,},](v7)(v8)
\end{tikzpicture}
\begin{tikzpicture}[scale=0.8]
\definecolor{cv0}{rgb}{0.0,0.0,0.0}
\definecolor{cfv0}{rgb}{1.0,1.0,1.0}
\definecolor{clv0}{rgb}{0.0,0.0,0.0}
\definecolor{cv1}{rgb}{0.0,0.0,0.0}
\definecolor{cfv1}{rgb}{1.0,1.0,1.0}
\definecolor{clv1}{rgb}{0.0,0.0,0.0}
\definecolor{cv2}{rgb}{0.0,0.0,0.0}
\definecolor{cfv2}{rgb}{1.0,1.0,1.0}
\definecolor{clv2}{rgb}{0.0,0.0,0.0}
\definecolor{cv3}{rgb}{0.0,0.0,0.0}
\definecolor{cfv3}{rgb}{1.0,1.0,1.0}
\definecolor{clv3}{rgb}{0.0,0.0,0.0}
\definecolor{cv4}{rgb}{0.0,0.0,0.0}
\definecolor{cfv4}{rgb}{1.0,1.0,1.0}
\definecolor{clv4}{rgb}{0.0,0.0,0.0}
\definecolor{cv5}{rgb}{0.0,0.0,0.0}
\definecolor{cfv5}{rgb}{1.0,1.0,1.0}
\definecolor{clv5}{rgb}{0.0,0.0,0.0}
\definecolor{cv6}{rgb}{0.0,0.0,0.0}
\definecolor{cfv6}{rgb}{1.0,1.0,1.0}
\definecolor{clv6}{rgb}{0.0,0.0,0.0}
\definecolor{cv7}{rgb}{0.0,0.0,0.0}
\definecolor{cfv7}{rgb}{1.0,1.0,1.0}
\definecolor{clv7}{rgb}{0.0,0.0,0.0}
\definecolor{cv8}{rgb}{0.0,0.0,0.0}
\definecolor{cfv8}{rgb}{1.0,1.0,1.0}
\definecolor{clv8}{rgb}{0.0,0.0,0.0}
\definecolor{cv9}{rgb}{0.0,0.0,0.0}
\definecolor{cfv9}{rgb}{1.0,1.0,1.0}
\definecolor{clv9}{rgb}{0.0,0.0,0.0}
\definecolor{cv0v1}{rgb}{0.0,0.0,0.0}
\definecolor{cv0v3}{rgb}{0.0,0.0,0.0}
\definecolor{cv0v5}{rgb}{0.0,0.0,0.0}
\definecolor{cv0v7}{rgb}{0.0,0.0,0.0}
\definecolor{cv0v8}{rgb}{0.0,0.0,0.0}
\definecolor{cv1v3}{rgb}{0.0,0.0,0.0}
\definecolor{cv1v4}{rgb}{0.0,0.0,0.0}
\definecolor{cv1v6}{rgb}{0.0,0.0,0.0}
\definecolor{cv1v7}{rgb}{0.0,0.0,0.0}
\definecolor{cv2v3}{rgb}{0.0,0.0,0.0}
\definecolor{cv2v4}{rgb}{0.0,0.0,0.0}
\definecolor{cv2v7}{rgb}{0.0,0.0,0.0}
\definecolor{cv2v8}{rgb}{0.0,0.0,0.0}
\definecolor{cv2v9}{rgb}{0.0,0.0,0.0}
\definecolor{cv3v8}{rgb}{0.0,0.0,0.0}
\definecolor{cv3v9}{rgb}{0.0,0.0,0.0}
\definecolor{cv4v5}{rgb}{0.0,0.0,0.0}
\definecolor{cv4v8}{rgb}{0.0,0.0,0.0}
\definecolor{cv4v9}{rgb}{0.0,0.0,0.0}
\definecolor{cv5v6}{rgb}{0.0,0.0,0.0}
\definecolor{cv5v7}{rgb}{0.0,0.0,0.0}
\definecolor{cv5v9}{rgb}{0.0,0.0,0.0}
\definecolor{cv6v7}{rgb}{0.0,0.0,0.0}
\definecolor{cv6v8}{rgb}{0.0,0.0,0.0}
\definecolor{cv6v9}{rgb}{0.0,0.0,0.0}
\Vertex[style={minimum
size=1.0cm,draw=cv0,fill=cfv0,text=clv0,shape=circle},LabelOut=false,L=\hbox{$10$},x=1.5cm,y=5cm]{v0}
\Vertex[style={minimum
size=1.0cm,draw=cv1,fill=cfv1,text=clv1,shape=circle},LabelOut=false,L=\hbox{$1$},x=0cm,y=4cm]{v1}
\Vertex[style={minimum
size=1.0cm,draw=cv2,fill=cfv2,text=clv2,shape=circle},LabelOut=false,L=\hbox{$2$},x=4cm,y=1.3cm]{v2}
\Vertex[style={minimum
size=1.0cm,draw=cv3,fill=cfv3,text=clv3,shape=circle},LabelOut=false,L=\hbox{$3$},x=1.5cm,y=1cm]{v3}
\Vertex[style={minimum
size=1.0cm,draw=cv4,fill=cfv4,text=clv4,shape=circle},LabelOut=false,L=\hbox{$4$},x=0cm,y=0cm]{v4}
\Vertex[style={minimum
size=1.0cm,draw=cv5,fill=cfv5,text=clv5,shape=circle},LabelOut=false,L=\hbox{$5$},x=1cm,y=3cm]{v5}
\Vertex[style={minimum
size=1.0cm,draw=cv6,fill=cfv6,text=clv6,shape=circle},LabelOut=false,L=\hbox{$6$},x=5cm,y=4cm]{v6}
\Vertex[style={minimum
size=1.0cm,draw=cv7,fill=cfv7,text=clv7,shape=circle},LabelOut=false,L=\hbox{$7$},x=3.5cm,y=5cm]{v7}
\Vertex[style={minimum
size=1.0cm,draw=cv8,fill=cfv8,text=clv8,shape=circle},LabelOut=false,L=\hbox{$8$},x=3cm,y=3cm]{v8}
\Vertex[style={minimum
size=1.0cm,draw=cv9,fill=cfv9,text=clv9,shape=circle},LabelOut=false,L=\hbox{$9$},x=5.0cm,y=0cm]{v9}
\Edge[lw=0.1cm,style={color=cv0v1,},](v0)(v1)
\Edge[lw=0.1cm,style={color=cv0v3,},](v0)(v3)
\Edge[lw=0.1cm,style={color=cv0v5,},](v0)(v5)
\Edge[lw=0.1cm,style={color=cv0v7,},](v0)(v7)
\Edge[lw=0.1cm,style={color=cv0v8,},](v0)(v8)
\Edge[lw=0.1cm,style={color=cv1v3,},](v1)(v3)
\Edge[lw=0.1cm,style={color=cv1v4,},](v1)(v4)
\Edge[lw=0.1cm,style={color=cv1v6,},](v1)(v6)
\Edge[lw=0.1cm,style={color=cv1v7,},](v1)(v7)
\Edge[lw=0.1cm,style={color=cv2v3,},](v2)(v3)
\Edge[lw=0.1cm,style={color=cv2v4,},](v2)(v4)
\Edge[lw=0.1cm,style={color=cv2v7,},](v2)(v7)
\Edge[lw=0.1cm,style={color=cv2v8,},](v2)(v8)
\Edge[lw=0.1cm,style={color=cv2v9,},](v2)(v9)
\Edge[lw=0.1cm,style={color=cv3v8,},](v3)(v8)
\Edge[lw=0.1cm,style={color=cv3v9,},](v3)(v9)
\Edge[lw=0.1cm,style={color=cv4v5,},](v4)(v5)
\Edge[lw=0.1cm,style={color=cv4v8,},](v4)(v8)
\Edge[lw=0.1cm,style={color=cv4v9,},](v4)(v9)
\Edge[lw=0.1cm,style={color=cv5v6,},](v5)(v6)
\Edge[lw=0.1cm,style={color=cv5v7,},](v5)(v7)
\Edge[lw=0.1cm,style={color=cv5v9,},](v5)(v9)
\Edge[lw=0.1cm,style={color=cv6v7,},](v6)(v7)
\Edge[lw=0.1cm,style={color=cv6v8,},](v6)(v8)
\Edge[lw=0.1cm,style={color=cv6v9,},](v6)(v9)
\end{tikzpicture}
    \caption{$L_1$ and $L_2$, Graphs that are $M$ cospectral for $M=A,L,Q,\cL,\D,\DL,\DQ$ and $\NDL$}
    \label{fig:cospecALL}
\end{center}
\end{figure}

\subsection{The number of graphs with a cospectral mate}\label{sec:cospecCompare}
The number of graphs that have cospectral mates has been computed for all graphs on 10 and fewer vertices for all matrices discussed in this paper (except for $\cL$, for which the number of graphs with a $\cL$-cospectral mates has only been computed for 9 and fewer vertices). These values are given in Tables \ref{tab:AdjCospecNum} and \ref{tab:DistCospecNum}. In Table \ref{tab:CospecPercent}, the percentage of graphs that have a cospectral mate is given for each matrix. It is obvious that $\NDL$ has significantly fewer graphs with a cospectral mate that any previously studied matrix on 10 or less vertices and we conjecture this pattern continues in to larger number of vertcies. This makes the normalized distance Laplacian a useful tool for determining if two connected graphs are isomorphic.


\begin{table}[h!]
    \centering
\begin{tabular}{c|c|c|c|c|c|c|c|c|c|c}
    $n$ & $\#$ graphs & $A$ & $L$ &$Q$ & $\cL$   \\
    \hline
    3&4&0&0&0&0\\
    4&11&0&0&2&2\\
    5&34&2&0&4&4\\
    6&156&10&4&16&14\\
    7&1,044&110&130&102&52\\
    8&12,346&1,722&1,767&1,201&201\\
    9&274,668&51,038&42,595&19,001&1,092\\
    10&12,005,168&2,560,516&1,412,438&636,607&
\end{tabular}
    \caption{Number of graphs with a cospectral mate with respect to each matrix. Counts for $A,L,Q$ from \cite{HS04}, counts for $\cL$ from \cite{BG11}.}
    \label{tab:AdjCospecNum}
\end{table}

\begin{table}[h!]
    \centering
\begin{tabular}{c|c|c|c|c|c|c|c|c|c|c}
    &$\#$ connected&  && &    \\
    $n$ &graphs& $\D$ & $\DL$&$\DQ$ & $\NDL$   \\
    \hline
    3&2&0&0&0&0\\
    4&6&0&0&0&0\\
    5&21&0&0&2&0\\
    6&112&0&0&6&0\\
    7&853&22&43&38&0\\
    8&11,117&658&745&453&2\\
    9&261,080&25,058&19,778&8,168&8\\
    10&11,716,571&1,389,984&787,851&319,324&7538
\end{tabular}
    \caption{Number of connected graphs with a cospectral pair with respect to each matrix. Counts for $\D,\DL,\DQ$ from \cite{AH18}.}
    \label{tab:DistCospecNum}
\end{table}

\begin{table}[h!]
    \centering
\begin{tabular}{c|c|c|c|c|c|c|c|c|c|c}
    $n$ & $A$ & $L$ &$Q$ & $\cL$& $\D$ & $\DL$&$\DQ$ & $\NDL$   \\
    \hline
    3&0 &0&0&0&0&0&0&0\\
    4&0&0&18.1818\%&18.1818\%&0&0&0&0\\
    5&5.8824\%&0&11.7647\%&11.7647\%&0&0&9.5238\%&0\\
    6&6.4103\%&2.5641\%&10.2564\%&8.9744\%&0&0&5.3571\%&0\\
    7&10.5354\%&12.4521\%&9.7701\%&4.9808\%&2.5791\%&5.0410\%&4.4549\%&0\\
    8&13.9478\%&14.3123\%&9.7278\%&1.6281\%&5.91886\%&6.7014\%&4.0748\%&0.0180\%\\
    9&18.5817\%&15.5078\%&6.9178\%&0.3976\%&9.5978\%&7.5755\%&3.1285\%&0.0031\%\\
    10&21.3284\%&11.7653\%&5.3028\%&&11.8634\%&6.7242\%&2.7254\%&0.0643\%
\end{tabular}
    \caption{Percent of total graphs (for $A,L,Q,\cL$) / connected graphs (for $\D,\DL,\DQ,\NDL$) that have a cospectral mate with respect to each matrix.}
    \label{tab:CospecPercent}
\end{table}

The similar matrix $T^{-1}\DL$ was used for all computations rather than $\NDL$ since it  does not include square roots and runs more quickly on {\em Sage}. 
To find $\NDL$-cospectral graphs graphs on 8 and fewer vertices, it was sufficient to use the {\em Sage} command cospectral.graphs(), which takes as its input any matrix defined with respect to a graph. However, this method was too computationally slow for 9 and 10 vertices. 

The method used for 9 and 10 vertices is a multi-step process that sorts the graphs in to groups of potentially cospectral graphs with an approximation of their characteristic polynomials using double precision decimal arithmetic. Each characteristic polynomial is evaluated at a large number, and then the floor and ceiling of the result is taken modulo a large prime number. The graphs are then sorted in to groups by this value. Using both the floor and ceiling is to ensure cospectral graphs end up in the same group at least once, despite any numerical approximation error. Within these groups, potential cospectral graphs are found by evaluating each approximated characteristic polynomial at a prime number, and searching for pairs of graphs for which this value is within an $\epsilon=0.00005$ tolerance. Then each of these pairs is checked for cospectrality using their exact characteristic polynomial.

\section{Concluding remarks}
In this paper we introduced the normalized distance Laplacian. In Section \ref{sec:NormDistLap} and derive bounds on it's eigenvalues. Most notably, we show $\partial^{\cL}<2$ for all graphs $G$ on $n\geq 3$ vertices, in contrast to normalized Laplacian, which has the property $\mu=2$ if and only if the graph is bipartite. It is natural to ask the following further questions: what is the maximum $\NDL$ spectral radius achieved by a graph on $n$ vertices and which graphs achieve it? Based on the behavior of the family $KPK_{n_1,n_2,n_3}$ for large $n_1,n_2,n_3$, we conjecture that the maximum $\NDL$ spectral radius tends to 2 as $n$ becomes large and that this value is achieved by some graph in the family $KPK_{n_1,n_2,n_3}$. We also found that the complete graph $K_n$ achieves the minimal spectral radius and conjecture that is the only such graph.


It would be interesting to find methods for constructing $\NDL$-cospectral graphs. Since in Section \ref{sec:Cospec} we show examples of $\DL$-cospectral constructions producing $\NDL$-cospectral graphs, it seemes likely that a suitable additional restriction placed on a $\DL$-cospectral construction may provide a $\NDL$-cospectral construction method. 


\end{document}